\documentclass[a4paper,12pt]{amsart}
\usepackage[utf8]{inputenc}
\usepackage{amsthm}
\usepackage{amsmath}
\usepackage{bm}
\usepackage{enumitem}
\usepackage{amsfonts}
\usepackage{amssymb}
\usepackage{mathtools}
\usepackage{appendix}
\usepackage{mathrsfs}
\usepackage{setspace}
\usepackage{comment}
\usepackage{xcolor}
\usepackage{todonotes}
\usepackage{thmtools} 
\usepackage[foot]{amsaddr}
\RequirePackage[a4paper,top=2.54cm,bottom=2.54cm,left=1.70cm,right=1.70cm,%
                headsep=1em,includehead,includefoot]{geometry}
\usepackage[pdfdisplaydoctitle,colorlinks,breaklinks,urlcolor=blue,linkcolor=blue,citecolor=blue]{hyperref} 
\usepackage[nameinlink,capitalise]{cleveref}

\providecommand{\U}[1]{\protect\rule{.1in}{.1in}}
\theoremstyle{plain}
\newtheorem{theorem}{Theorem}[section]

\newtheorem{lemma}[theorem]{Lemma}
\newtheorem{prop}[theorem]{Proposition}

\theoremstyle{definition}
\newtheorem{definition}{Definition}[section]

\newtheorem{remark}{Remark}[section]

\newtheorem*{maintheorem*}{Main Theorem}
\newtheorem*{maincorollary*}{Main Corollary}

\newcommand{\R}{\ensuremath{\mathbb{R}}}

\newcommand{\E}{\ensuremath{\mathbb{E}}}

\def\V{V}

\def\Vprime{V^\prime}

\title[LDP for a T-monotone obstacle problem  with  multiplicative noise]
{Large deviations for an obstacle problem	with T-monotone	operator and  multiplicative noise }
\author[Y. Tahraoui]{Yassine Tahraoui}
\address{Scuola Normale Superiore, Piazza dei Cavalieri, 7, 56126 Pisa, Italia}
\email{\href{mailto:yassine.tahraoui at sns.it}{yassine.tahraoui at sns.it}}
\date\today

\begin{document}

\begin{abstract}
	In this paper, we	study the	large deviation principle	(LDP) for obstacle problems governed	by	a	T-monotone	operator and	small multiplicative	stochastic	reaction.	Our approach	  relies on a combination of new sufficient condition to	prove	LDP	by	Matoussi, Sabbagh and Zhang [\textit{Appl. Math. Optim.}, 2021]	and	Lewy-Stampacchia	inequalities	to	manage	the		Lagrange
	multiplier	associated	with	the	obstacle.
\end{abstract} 
	\maketitle
\section{Introduction}
We are interested in investigating	an	 asymptotic properties of the	solution to	a	stochastic obstacle	problems with small	noise.	More precisely,	our	aim	is	to	prove the  	large	deviation	principle	for the	solution  $(u_\delta)_{\delta\downarrow	0}$ to some obstacle problems which can be written  (see	Section	\ref{sec-formulation})
\begin{align}\label{obstacle-pb}
 \partial I_{K}(u_\delta)\ni f-\partial_t \Big( u_\delta-\displaystyle\delta\int_0^\cdot G(u_\delta,\cdot)dW\Big)- A(u_\delta,\cdot);	\quad	\delta>0,  
\end{align}
where $K$ is a closed convex subset of  $L^p(\Omega;L^p(0,T;V))$ related to a  constraint $\psi$, $A$ is a nonlinear T-monotone operator defined on some Banach	space $V$, $(\Omega,\mathcal{F}, (\mathcal{F}_t)_{t\geq 0}, P)$ is a filtered probability space  and  $W(t)$ is a $Q$-Wiener process in some separable Hilbert space $H.$\\

Obstacle type problems 	\eqref{obstacle-pb} can be used to describe several  phenomena		and	arise in many mathematical models associated with a number of fields including	Physics, Biology and Finance.	For example, 	 the evolution of damage in a continuous	medium taking into account the microscopic description \cite{Bauz},		the	question	of	 American option	pricing,	some	questions	of	flows in porous media	and phase transition questions,	some optimal stopping-time problems	and		population dynamics	models,	see \textit{e.g.}	\cite{Bauz,Bens-lions,RoK13,Yas}	and	references	theirin.	In the deterministic setting, it can be formulated by using variational inequalities,	see \textit{e.g.} \cite{GMYV,GYV,	Mok-Yas-Val,Yas} and their references.
Concerning 	the	well-posedness	questions, many authors have been interested in the topic, without	seeking to be exhaustive, let us mention  \cite{DonatiMartinPardoux}	about	the	reflected solutions of parabolic SPDEs driven by a space time	white noise	on the spatial	interval	$[0,1]$,			quasilinear			stochastic	PDEs	with	obstacle	in		\cite{DenisMatoussiZhang}	and	\cite{Hau-Par}	for	stochastic	variational	inequalities.	Besides	the	well-posedness,	a	regularity	properties	of	the	reflected	stochastic	measure		was	studied	in	\cite{YV22},	where stochastic Lewy-Stampacchia’s inequalities were proved	in	the	case	of	stochastic  T-monotone obstacle problems	and	then	\cite{IYG,NYGA23}	for	 stochastic scalar conservation laws with constraint	and		pseudomonotone	parabolic	obstacle	problems.	Recently, the existence and uniqueness of ergodic invariant measures have been studied in	\cite{Yas24}. 
\\

Large	deviation	principle	concerns	the	exponential	decay,	with	the	corresponding rate	 functions,	of	the laws of solutions 	to	a	different	stochastic	dynamics.	There	exist	many	works	on	large	deviations	for	dynamical systems	in	finite/infinte	dimensional	setting	driven	by	 white noises.
Without	exaustivness,	let	us	refer	to	the	work	of	Freidlin	and	Wentzell	\cite{Freid-Went}	for	the	small	perturbation	for	stochastic	differential	equations,		\cite{BD00}	
			about	the	variational	representation		for	functionals	of	Wiener	process	and	its	application	to	prove	large	deviation	principle	and	\cite{BD08,Daprato-Zab92,Hong21,Liu10,Liu20,Ren-Zhang08}	for	large	deviations	results	about	different	stochastic	dynamics	related	with	stochastic	equations.\\
			
	We	recall	that	\eqref{obstacle-pb}	is	a	free	boundary	type	problem		and	also		can	be	understood	as	coupled	system	of	 stochastic	PDEs	as	well,	which	generates	many	mathematical	difficulties	in	the	analysis	of	a	such	problem.		In	particular,		the	management	of	the	singularities	caused	by	the	obstacle.	Therefore,	one	needs	to	manage	in	an	appropriate	way	the	stochastic	reflected measure,	resulting	from	the	interaction	between	the	solution	and		the	obstacle	near	to	the	contact	set.	
	Concerning		large	deviations	of	obstacle	problem,	let	us	mention		\cite{XU09}	for
	 large deviation principle to SPDEs with reflection,			studied	in	\cite{DonatiMartinPardoux},	by	using	the	 weak convergence approach.	Another	type	of	result	concerns	large deviation principle for invariant measures to solutions of reflected	SPDEs  driven by a space–time white noise	in	1D,	see	\cite{Zhang12}.		Recently,	the	authors	in	\cite{Matou21}	provided a new	sufficient condition to
	verify the criteria of	 the weak convergence	and	used	it		to	study	large	deviation	principle	for	quasilinear	stochastic	obstacle	problems.
	Our	aim	is	to	present	a	result	about	large	deviations		  for stochastic	obstacle	problem	governed 	by	 T-monotone	operator,	including	p-Laplace	type	operator,	a	multiplicative	noise	and	large	class	of	obstacles.	The	method	is	based	on	the	combination	of			proving	the	new	sufficient	condition	of	\cite[Theorem	3.2]{Matou21}	in	our	setting	and	using	Lewy-Stampacchia	inequalities	to	manage	the	stochastic	reflected measure.	\\

\subsubsection*{Structure of the paper}	After giving the hypotheses, we	recall 	some	basics	about		stochastic	obstacle	problems	and	large	deviations	in	Section	\ref{Section-notation}.	Section	\ref{sec-formulation}	is	devoted	to	the	mathematical	formulation	of	the	problem	and	the	main	results	of	this	paper.	In	Section	\ref{Sec-Skeleton},	we	prove	the	well-posedness	and	the	corresponding	Lewy-Stampacchia	inequalities	for	the	skeleton	equation	\eqref{Skeleton},	by	using	penalization	and	approximation	techniques.	Section	\ref{section-LDP-proof}	concerns	the	proof	of	the	large	deviation	principle	for	the	stochastic	dynamics	associated	with	\eqref{obstacle-pb},	where	Lewy-Stampacchia	inequalities	will	be	used	to	manage	the	singularities	caused	by	the	obstacle,	which	serves	to	study	the	continuity	of	the		 	skeleton equations	\eqref{Skeleton}	from	the	 Cameron–Martin Hilbert space	endowed	with	the	weak	topology	to  some	Polish	space	endowed	with	the	strong	topology.	Finally,	we	study	the	limit	as	$\delta\to0$	in	appropriate	sense.

\section{Content of the study}\label{Section-notation}
\subsection{Notation and  functional spaces}
Let us denote by $D \subset \R^d,d\geq1$ a Lipschitz bounded domain, $T>0$ and consider $\max(1,\frac{2d}{d+2}) <p <+\infty$.  
As usual, $p^\prime=\frac{p}{p-1}$ denotes the conjugate exponent of $p$, $V=W^{1,p}_0(D)$, the sub-space of elements of $W^{1,p}(D)$ with null trace, endowed with Poincar\'e's norm,  $H=L^2(D)$ is identified with its dual space so that, the corresponding dual spaces to $V$, $V^\prime$ is $W^{-1,p^\prime}(D)$ and the Lions-Guelfand triple $\V \hookrightarrow_d H=L^2(D)   \hookrightarrow_d \Vprime$ holds. 
Denote by $p^*=\frac{pd}{d-p}$ if $p<d$ the Sobolev embedding exponent and remind that 
\begin{align*}
\text{ if }p<d,\quad &V \hookrightarrow L^a(D),\ \forall a \in [1,p^*]\text{ and compactly if }a \in [1,p^*),
\\
\text{ if }p=d,\quad & V \hookrightarrow L^a(D),\ \forall a <+\infty\text{ and compactly}, 
\\
\text{ if }p>d,\quad &V \hookrightarrow C(\overline D)\text{ and compactly }.
\end{align*}
Since $\max(1,\frac{2d}{d+2}) <p <+\infty$, the compactness of the embeddings hold in Lions-Gelfand triple.	The duality bracket for $T \in V^\prime$ and $v \in V$ is denoted $\langle T,v\rangle$ and the scalar product in $H$ is denoted by $(\cdot,\cdot)$.
\\
 
 Let $(\Omega,\mathcal{F},P)$ be  a complete probability space  endowed with a  right-continuous filtration $\{\mathcal{F}_t\}_{t\geq 0}$\footnote{For	example,	$(\mathcal{F}_t)_{t\geq 0}$ is the augmentation of the filtration generated by  $\{W(s), 0 \leq s\leq t \}_{0\leq t\leq T}$.} completed with respect to the measure $P$.  $W(t)$ is	a	$\{\mathcal{F}_t\}_{t\geq 0}$-adapted 	$Q$-Wiener process in $H$,	where $Q$	is 	non-negative symmetric   operator  with	finite	trace	\textit{i.e.} $tr Q < \infty$. Denote by $\Omega_T=(0,T)\times \Omega$ and $\mathcal{P}_T$ the predictable $\sigma$-algebra on $\Omega_T$\footnote{$\mathcal{P}_{T}:=\sigma(\{ ]s,t]\times F_s \vert 0\leq s < t \leq T,F_s\in \mathcal{F}_s \} \cup \{\{0\}\times F_0 \vert F_0\in \mathcal{F}_0 \})$ (see \cite[p. 33]{RoK15}). Then, a process defined on $\Omega_T$ with values in a given space $E$ is predictable if it is $\mathcal{P}_{T}$-measurable.}. 
Set $H_0=Q^{1/2}H$ and we recall that $H_0$ is a separable Hilbert space endowed with the inner product $ (u,v)_0=(Q^{-1/2}u,Q^{-1/2}v)$, for any $u,v\in H_0$. The space	$(L_2(H_0,H), \Vert \cdot \Vert_{L_2(H_0,H)})$ stands	for the space of Hilbert-Schmidt operators from $H_0$ to $H$\footnote{	$\Vert T \Vert_{L_2(H_0,H)}^2=\displaystyle\sum_{k\in\mathbb{N}}\Vert	Te_k\Vert_H^2$	where	$	\{e_k\}_{k\in\mathbb{N}}$	is	an	orthonormal basis	for	$H_0$,	see	\textit{e.g.}	\cite[Appendix B]{RoK15}.},	$\E$	stands	for	the	expectation	with	respect	to	$P$. \\

 We recall that an element $\xi\in L^{p'}(0,T;V')$	(resp.	$L^{p'}(\Omega_T;V')$) is called \textit{non-negative} \textit{i.e.}	$\xi\in ( L^{p^\prime}(0,T;\Vprime))^+$(resp.	$(L^{p'}(\Omega_T;V'))^+$)	iff 
$$	\int_0^T\langle \xi,\varphi\rangle_{V',V}\,dt\geq 0\quad	(\text{	resp.	}\E\int_0^T\langle \xi,\varphi\rangle_{V',V}\,dt\geq 0)$$
holds for all $\varphi\in L^p(0,T;V)$ (resp.	$L^{p}(\Omega_T;V)$)	such that $\varphi\geq 0$. In this case, with a slight abuse of notation, we will often write $\xi\geq 0$.\\
Denote by 	$ L^{p}(0,T;V)^* =( L^{p^\prime}(0,T;\Vprime))^+ - (L^{p^\prime}(0,T;\Vprime))^+ \varsubsetneq L^{p^\prime}(0,T;\Vprime)$
 the order dual as being the difference of two non-negative elements of $L^{p^\prime}(0,T;\Vprime)$, more precisely $h\in L^{p}(0,T;V)^*$ iff
$h=h^+-h^-$ with $h^+,h^- \in L^{p^\prime}(0,T;\Vprime)^+$.

\subsection{Assumptions}
We will consider in the sequel the following assumptions:
\begin{enumerate}
	\item[H$_1$] :
	Let $A: V\times [0,T] \to V^\prime$, \ $G: H\times  [0,T] \to L_2(H_0,H)$, $\psi : [0,T] \to V$, $f : [0,T] \to V^\prime$	such that:
	\begin{align*}
 \text{ for any	} v\in V \text{	and	} u \in H:\quad A(v,\cdot),	G(u,\cdot),	\psi \text{	and	}	f	\text{	 are measurable. }
	\end{align*}

	\item[H$_2$] : $\exists \alpha,\bar K >0,  \lambda_T, \lambda \in \mathbb{R}$, $l_1\in L^\infty([0,T])$ and $g\in L^{\infty}([0,T])$ such that:
	\begin{enumerate}
		\item[H$_{2,1}:$]   $t\in [0,T]$ a.e.,\quad  $\forall  v \in V,\quad \langle A(v,t),v\rangle+\lambda \Vert v\Vert_H^2+l_1(t) \geq \alpha \Vert v\Vert_V^p.$
		\item[H$_{2,2}:$]  ($T-$ monotonicity)   \quad   $t\in [0,T]$ a.e.,\quad  $\forall v_1,v_2 \in V$, 
		\begin{align*}
		\lambda_T(v_1-v_2,(v_1-v_2)^+)_H+\langle A(v_1,t)-A(v_2,t),(v_1-v_2)^+\rangle \geq 0.
		\end{align*}
		Note that since $v_1-v_2 = (v_1-v_2)^+ - (v_2-v_1)^+$,  $\lambda_T Id+A$ is also monotone.
		\item[H$_{2,3}:$]   $t\in [0,T]$ a.e.,\quad  $\forall  v \in V,\quad  \Vert A(v,t)\Vert_{V^\prime} \leq \bar K\Vert v\Vert_V^{p-1}+g(t).$
		\item[H$_{2,4}:$]   (Hemi-continuity) \quad    $t\in [0,T]$ a.e.,\quad  $\forall v,v_1,v_2 \in V$,\\ \hspace*{3cm} $ \eta \in \R \mapsto \langle A(v_1+\eta v_2,t),v\rangle$ is continuous.
	\end{enumerate}
	%
	\item[H$_3$] : $\exists  M,L >0$  such that
	\begin{enumerate}
		\item[H$_{3,1}:$]  $t\in [0,T]$ a.e.,\quad  $\forall  \theta,\sigma \in H$, \quad  $\Vert G(\theta,t)-G(\sigma,t)\Vert_{L_2(H_0,H)}^2 \leq M \Vert \theta-\sigma\Vert_H^2$.
		\item[H$_{3,2}:$]  $t\in [0,T]$ a.e.,\quad  $\forall  u \in H$, \quad  $ \Vert G(u,t)\Vert_{L_2(H_0,H)}^2 \leq L (1+\Vert u\Vert_H^2).$
	\end{enumerate}
	\item[H$_4$] : $\psi \in  L^\infty(0,T;V),\  \frac{d \psi}{dt}  \in  L^{p^\prime}(0,T;\Vprime)$ and	$G(\psi,t)=0$ a.e.	$t\in [0,T]$.
	\item[H$_5$] :  $f\in L^{\infty}(0,T;V^\prime)$ and  assume  that 
$
h^+-h^-=h = f - \partial_t \psi  - A(\psi,\cdot) \in L^{p}(0,T;V)^*.
$		
	\item[H$_6$] :  $u_0 \in H$ satisfies the constraint, \textit{i.e.} $u_0 \geq \psi(0)$.
	\item[H$_{7}$]	:  (Strong monotonicity)  $\exists \lambda_T\in \mathbb{R}, \exists \bar{\alpha} >0$:   $t\in [0,T]$ a.e.,\quad  $\forall v_1,v_2 \in V$, 
	\begin{align*}
	\langle A(v_1,t)-A(v_2,t),v_1-v_2\rangle \geq \bar{\alpha}\Vert v_1-v_2\Vert_V^p- \lambda_T\Vert v_1-v_2\Vert_H^2.
	\end{align*}
	
\end{enumerate}
\begin{remark}
	Note	that	$G(\psi,t)=0$	a.e.	in	$[0,T]$	in	
	$H_4$		plays	a	crucial	role	in	the	analysis		of	skeleton	equation	\eqref{Skeleton}.	More	precisely,	it	serves	in	the	proof	of	Lewy-Stampacchia	inequalities	\eqref{LS-Skeleton}	associated		with		\eqref{Skeleton}	"deterministic",	by	using	the	natural	assumption	$H_5$	(see	\textit{e.g.}	\cite[Thm.	2.2]{GMYV}),	to	manage	the	reflected	measure	related	to	the	obstacle	in	appropriate	sense.	Moreover,	a	large	class	of	$V$-valued	obstacle 	problems	  can reduce to
	the question of a positivity obstacle problem with a stochastic reaction term vanishing	at $0$,			see	\cite[Rmk.	3]{YV22}.
\end{remark}
\begin{remark}
	\begin{itemize}
		\item[(i)]		In	the	case	$p\geq2$,	the	assumptions	on		$l_1,g,
		\psi $ and
		$f$	can	be	relaxed	to		weaker	assumptions,	more	precisely,	we	can	consider	$l_1\in L^1([0,T]),	g\in	L^{p^\prime}([0,T])$ in	H$_2$,
		$\psi \in  L^p(0,T;V)$	in	H$_4$ and
		$f\in L^{p^\prime}(0,T;V^\prime)$	in	$H_5$.\\
		
				\item[(ii)]	The	assumptions	$l_1,g\in L^\infty([0,T])$ in	H$_2$,
		$\psi \in  L^\infty(0,T;V)$	in	H$_4$ and
			  $f\in L^{\infty}(0,T;V^\prime)$	in	$H_5$	are	assumed	to	ensure	the	existence	of	solution	to	the	Skeleton	equations	\eqref{Skeleton}	if	$\max(1,\frac{2d}{d+2})<p<2$.	Else,		extra	assumptions	on	$G$		has	to	be	considered,		we	refer	to	\textit{e.g.}	\cite[Section	2]{Liu10}	for	such	assumptions.  
					  \\
		  \item[(iii)] 	In	the	case		$1<p\leq\dfrac{2d}{d+2}	(d\geq	3)$,		one	needs	either	to	use	a method of some finite dimensional approximation of
		  the diffusion $G$(to	establish	large deviation principle	in	the	case	of	equations),	which	leads	to	an	extra	assumption	on	$G$	\cite[Section	2	(A4)]{Liu10}	or	some	temporal regularity on	$G$,	see	\textit{e.g.}	\cite[Section	2	(A4)]{Liu20}.	The	situation		seems	more	delicate		because	of	the	reflected	measure,	which	depends	on	the	solution	as	well.\\
		  
		\item[(iv)]	The	role	of	$H_7$	is		to	establish	large deviation principle	in		$L^p(0,T;V)\cap	C([0,T];H)$.	Else,	we	obtain		large deviation principle	only	in	$C([0,T];H)$,		see	Theorem	\ref{Main-Thm-LDP}.

	\end{itemize}
\end{remark}

\subsection{Stochastic obstacle problems} \label{section2}

	Denote by $K$ the convex  set of admissible functions
	$$K=\{v \in L^p(\Omega_T;V),  \quad v(x,t,\omega)\geq \psi(x,t) \quad \text{a.e. in } D\times\Omega_T\}.$$

The	question	is to	find $(u,k),$ in a space defined straight after, solution to
\begin{align}\label{I}
\begin{cases}
du+A(u,\cdot)ds+ kds= fds+G(u,\cdot)dW \quad & \text{in} \quad D\times  \Omega_T, \\ 
u(t=0)=u_0 & \text{in} \quad H,\  a.s.,\\ 
u \geq \psi & \text{in} \quad  D\times  \Omega_T, \\  
u=  0   &\text{on} \quad\partial D\times \Omega_T,  \\
\langle k,u-\psi\rangle  = 0 \hspace*{0.2cm} \text{and} \hspace*{0.2cm} k \leq 0  & \text{in} \quad  \Omega_T.
\end{cases}
\end{align}
\begin{definition}\label{def}
	The pair $(u,k)$ is a solution to Problem \eqref{I} if:
	\begin{itemize}
		\item[$\bullet$]  $u:\Omega_T\to H$ and $k:\Omega_T \to V^\prime$ are predictable	processes, $u \in L^2(\Omega;C([0,T];H))$.
		\item[$\bullet$]  	 $u\in L^p(\Omega_T;V)$, $u(0,\cdot)=u_0$ and $u\geq \psi$, \textit{i.e.}, $u\in K$.
			\item  $k \in L^{p^\prime}(\Omega_T; V^\prime)$ with 
		\begin{align}\label{230530_01}
		- k\in (L^{p^\prime}(\Omega_T;V^\prime))^+ \ \text{and} \ \E{\int_0^T\langle k,u-\psi\rangle_{V',V}\,dt}=0.
		\end{align}
			\item[$\bullet$] P-a.s, for all $t\in[0,T]$: 
		$u(t)+\displaystyle\int_0^t [A(u,\cdot)+k]ds=u_0+\displaystyle\int_0^t G(u,\cdot)dW(s)+\displaystyle\int_0^t fds$	in	$V^\prime$.
		
	\end{itemize}
\end{definition}

\begin{remark}
	Condition \eqref{230530_01} can be understood as a minimality condition on $k$ in the sense that $k$ vanishes on the set $\{u>\psi\}$. Moreover, \eqref{230530_01} implies that, for all $v \in K$,  
	\[\E{\int_0^T\langle k,u-v\rangle_{V',V}\,dt}  \geq 0.\]
\end{remark}
\begin{remark}
	Note that  Problem \eqref{I} can be written in the equivalent form (see \cite[p.$7-8$]{BARBU}):
	\begin{align*} 
	\partial I_{K}(u)\ni f-\partial_t \left(u-\int_0^\cdot G(u)\,dW\right) - A(u,\cdot),
	\end{align*}
	where $\partial I_{K}(u)$ represents the sub-differential of $I_{K}: L^p(\Omega_T;V)\rightarrow \overline{\mathbb{R}} $ defined as 
	\begin{align*}
	I_{K}(u)=\left\{
	\begin{array}{l}
	0, \quad\quad  u \in K,\\ [1.2ex]
	+\infty, \quad u \notin K,
	\end{array}
	\right.
	\end{align*}
	and 
		$\partial I_{K}(u)=N_{K}(u):=
	\left\{y \in L^{p^\prime}(\Omega_T;V^\prime); \ \E\displaystyle\int_0^T\langle y,u-v\rangle_{V',V}dt \geq 0, \ \text{for all} \ v \in K \right\}.$
	
\end{remark}
\begin{theorem}\cite[Theorem	1]{YV22}\label{TH1}
	Under Assumptions (H$_1$)-(H$_6$), there exists  a unique solution $(u,k)$ to Problem \eqref{I} in the sense of Definition \ref{def}. Moreover,  the following  Lewy-Stam\-pac\-chia's inequality\footnote{See	\cite[Definition 1.3.]{NYGA23}.} holds
	 $0\leq  \partial_t \Big( u-\displaystyle\int_0^\cdot G(u,\cdot)dW\Big)+ A(u,\cdot)-f \leq h^-=  \left(f - \partial_t  \psi  - A(\psi,\cdot)\right) ^-.$
	\end{theorem} 

\subsection{Large	deviations}
Let $E$ be a Polish space with the Borel $\sigma-$field $\mathcal{B}(E)$.
\begin{definition}\cite[Def. 4.1]{BD00}(Rate function)
	A function $I:E\to [0,\infty]$ is called a 	rate function	 on $E$ if  for each $M <\infty$ the level set $\{ x\in E: I(x) \leq M\}$ is a compact subset.
\end{definition}
\begin{definition}
		(Large deviation principle) A family $\{ X^\epsilon\}_\epsilon$ of $E-$valued random elements is said to satisfy a large deviation principle  with rate function $I$ if for each Borel subset $G$ of $E$
		$$ -\inf_{x\in G^{0}} I(x)\leq \liminf_{\epsilon\to 0}\epsilon^2\log P(X^\epsilon \in G)\leq  \limsup_{\epsilon\to 0}\epsilon^2\log P(X^\epsilon \in G) \leq -\inf_{x\in \bar{G}} I(x),$$
		where $G^0$ and $\bar{G}$ are respectively the interior and the closure of $G$ in $E$.	
	\end{definition}

\begin{definition}\cite[Def. 4.2]{BD00}
	Let $I$ be a rate function on $E$. A family $\{ X^\delta\}_{\delta>0}$ of E-valued random elements is said to satisfy the Laplace principle   on $E$ with rate function $I$ if for all real valued, bounded, and continuous functions $h$ on $E$:
	\begin{equation}
		\lim_{\epsilon\to 0}\epsilon\log \E\big\{ \exp[-\dfrac{1}{\epsilon}h(X^\epsilon)]\big\}=-\displaystyle\inf_{x\in E}\{ h(x)+I(x)\}.
	\end{equation}
\end{definition}
We	recall	that the large deviation principle 	and	the Laplace principle  are equivalent	for  $E$-valued random elements,	see	\textit{e.g.}	\cite{BD00}.\\

Now,	let us introduce some notations to formulate LDP.
\begin{align*}
	&\text{	For	}  N \in \mathbb{N}:	\quad	S_N= \Big\{\phi \in L^2(0,T;H_0); \quad \int_0^T\Vert \phi(s)\Vert_{H_0}^2ds \leq N\Big\},\\	
	&\mathcal{A}=\Big\{v: \quad v \text{ is }  H_0\text{-valued }\text{ predictable process and } P\big( \int_0^T\Vert v(s)\Vert_{H_0}^2ds < \infty \big)=1\Big\}.
\end{align*}

Recall	that	 the	set $S_N$ endowed with the weak topology is a Polish space	and	define
\begin{align}\label{An-LDP-2}
\mathcal{A}_N=\{ v\in \mathcal{A}: v(\omega)\in S_N \text{ P-a.s.}\}.
\end{align}

Let $g^\delta:C([0,T],H) \to E$ be a measurable map such that $g^\delta(W(\cdot))=X^\delta.$

\subsubsection*{\underline{Assumption	(H)}}
Suppose	that	 there	exists a  measurable map $g^0:C([0,T];H) \to E$ such that	
\begin{enumerate}
	\item	For	every	$N<\infty$,	any	family	$\{ v^\delta: \delta>0\} \subset \mathcal{A}_N $	and	any	$\epsilon>0$,
	$$\lim_{\delta\to 0}P(\rho(Y^\delta,Z^\delta)>\epsilon)=0,$$
	where	$Y^\delta=	g^\delta(W(\cdot)+\dfrac{1}{\delta}\int_0^\cdot v^\delta_sds),	Z^\delta=g^0(\int_0^\cdot v_s^\delta	ds)$	and	$\rho(\cdot,\cdot)$	stands	for	the	metric	in	the	space	$E$.
	\item For	every	$N<\infty$	and	any	family $\{ v^\delta: \delta>0\} \subset S_N $	satisfying	that	$v^\delta$	converges	weakly	to	some	element	$v$	as	$\delta\to	0$,	$g^0(\int_0^\cdot v_s^\delta	ds)$	converges	to	$g^0(\int_0^\cdot v_s	ds)$	in	the	space	$E$.
\end{enumerate}

Let	us	recall	the	following	result	from	\cite{Matou21},	which	will	be	used	to	prove	Theorem	\ref{Main-Thm-LDP}.
\begin{theorem}\label{LDP-classic-result}	\cite[Theorem	3.2]{Matou21}	
	If  $X^\delta=g^\delta(W(\cdot))$ and the Assumption (H) holds, then the family $\{X^\delta\}_{\delta>0}$ satisfies the  Large deviation principle on $E$ with the  rate function $I$ given by
	\begin{equation}\label{rate-fct}
		I(f)=\displaystyle\inf_{\{\phi \in L^2(0,T;H_0); f=g^0(\int_0^\cdot\phi_sds)\}}\big\{\dfrac{1}{2}\int_0^T\Vert \phi(s)\Vert_{H_0}^2ds\big\},
	\end{equation}
	with	the	convention	$\inf\{\emptyset\}=\infty.$
	\end{theorem}
\section{Formulation	of	the	problem	and	main	results}\label{sec-formulation}
\subsection{Obstacle problem with small noise}

Let	$\delta>0$,	we are concerned with the small noise large deviation principle	(LDP)
of the following	 problem:
\begin{align}\label{I-LDP}
\begin{cases}
du_\delta+A(u_\delta,\cdot)ds+ k_\delta ds= fds+\delta G(\cdot, u_\delta)dW \quad & \text{in} \quad D\times  \Omega_T, \\ 
u_\delta(t=0)=u_0 & \text{in} \quad D,\\ 
u_\delta \geq \psi & \text{in} \quad  D\times  \Omega_T, \\  
u_\delta=  0   &\text{on} \quad\partial D\times \Omega_T,  \\
\langle k_\delta,u_\delta-\psi\rangle_{V^\prime,V}  = 0 \hspace*{0.2cm} \text{and} \hspace*{0.2cm} -k_\delta \in	(V^\prime)^+  & \text{a.e.	in} \quad  \Omega_T.
\end{cases}
\end{align}

Thanks to Theorem \ref{TH1},  there exists  a unique solution  $X_\delta:=(u_\delta,k_\delta)$ to Problem \eqref{I-LDP} in the sense of Definition \ref{def}. In particular, $\{X_\delta\}_\delta$ takes values in $$(C([0,T];H)\cap L^p(0,T;V))\times L^{p^\prime }(0,T;V^\prime) \quad P-a.s.$$
It's well-known that $(C([0,T];H)\cap L^p(0,T;V),\vert \cdot\vert_T)$ is a Polish space with the following norm
$$ \vert f-g\vert_T=\displaystyle\sup_{t\in[0,T]}\Vert f-g\Vert_H+\Big(\int_0^T \Vert f-g\Vert_V^pds\Big)^{1/p}. $$
The existence of a unique strong solution of the obstacle problem	\eqref{I-LDP}	ensures	the	existence	 of	a Borel-measurable map	(see  \textit{e.g.	}\cite[Thm. 2.1 ]{RoK08}) 
\begin{align}\label{def-g-delta}
	g^\delta:C([0,T];H)\to C([0,T];H)\cap L^p(0,T;V)
\end{align}
such that $u_\delta=g^\delta(W)$ P-a.s.  
	In	order	to	state	the	main	results,	let	us	introduce		the skeleton
	equation associated to	\eqref{I-LDP}.	 	Let $\phi \in L^2(0,T;H_0)$,	we	are	looking	for	$(y^\phi,R^\phi)$	solution	to	 \begin{align}\label{Skeleton}
	\begin{cases}
	\dfrac{dy^\phi}{dt}+A(y^\phi,\cdot)+ R^\phi= f+G(\cdot, y^\phi)\phi    \quad  \text{ in } V^\prime, y^\phi(0)=u_0 \text{ in } H ,\\ 
	y^\phi \geq \psi  \text{ a.e. in } \quad  D\times [0,T], \quad -R^\phi \in (V^\prime)^+: 
	\langle R^\phi,y^\phi-\psi\rangle_{V^\prime,V}  = 0	\text{ a.e. in }	[0,T]. \quad
	\end{cases}
	\end{align}

	\begin{definition}\label{Def-Skelton} 
				The pair  $(y^\phi,R^\phi)$ is a solution to \eqref{Skeleton} if	and	only	if:
		\begin{itemize}
			\item[$\bullet$]  $(y^\phi,R^\phi) \in (L^p(0,T;V)\cap C([0,T];H))\times L^{p^\prime}(0,T;V^\prime) $, 
			$y^\phi(0)=u_0 $ and $ y^\phi \geq \psi .$
			\item[$\bullet$]  $-R^\phi	\in(L^{p^\prime}(0,T;V^\prime))^+$	and		$ \displaystyle\int_0^T\langle R^\phi,y^\phi-\psi\rangle ds  = 0$ . 
			\item[$\bullet$] For all $t\in[0,T]$, 
			$$(y^\phi(t),\Phi)+\displaystyle\int_0^t\langle R^\phi,\Phi\rangle ds+\displaystyle\int_0^t \langle A(y^\phi,\cdot), \Phi\rangle ds=(u_0,\Phi)+\int_0^t(f+G(\cdot, y^\phi)\phi, \Phi) ds,\quad\forall	\Phi\in	V.$$ 
		\end{itemize}
				\end{definition}
			\begin{remark}	If	$(y^\phi,R^\phi)$	is	a	solution	to	\eqref{Skeleton}	in	the	sense	of	Definition	\ref{Def-Skelton}.	Then,	$y^\phi$	satisfies	the	following	variational	inequality:		for	any	$v \in L^p(0,T;V)$	such	that	$v\geq \psi$
				\begin{align*}
				\int_0^T\langle	\dfrac{	dy^\phi}{ds}+A(\cdot,y^\phi)-G(\cdot, y^\phi)\phi-f,v-y^\phi\rangle	ds\geq	0.
				\end{align*}
			\end{remark}
			
			\begin{prop}\label{Prop-Skeleton}
				Under Assumptions H$_1$-H$_6$, there exists  a unique solution $(y^\phi,R^\phi)$ to  \eqref{Skeleton} in the sense of Definition \ref{Def-Skelton}.	Moreover,	the following  Lewy-Stampacchia inequality holds: 
				\begin{align}\label{LS-Skeleton}
				0\leq  \dfrac{	dy^\phi}{ds}+ A(y^\phi,\cdot)- G(y^\phi,\cdot)\phi-f \leq h^-=  \left(f - \partial_t  \psi  - A(\psi,\cdot)\right) ^-.
				\end{align}
			\end{prop}
		\begin{proof}
		The	proof	of	Proposition	\ref{Prop-Skeleton}	will	be	given	in	Section	\ref{Sec-Skeleton}.		\end{proof}
		Now,	let	us		
				define	the		measurable	mapping	$g^0$	as	follows
					\begin{align}\label{g0-def}
			g^0:	C([0,T];H)	&\to	L^p(0,T;V)\cap	C([0,T];H)\notag\\
		\int_0^\cdot\phi	ds&\mapsto	g^0(\int_0^\cdot\phi	ds):=y^\phi\quad	\text{ for		}	\phi \in L^2(0,T;H_0),			\end{align}
			where	$y^\phi\in	L^p(0,T;V)\cap	C([0,T];H)$	is	the	unique	solution	of	\eqref{Skeleton}.	Now,
			
			introduce	the	following	rate	function
			\begin{align}\label{rate-fctn}
				I(y)=\displaystyle\inf_{\{\phi \in L^2(0,T;H_0) \}}\big\{\dfrac{1}{2}\int_0^T\Vert \phi(s)\Vert_{H_0}^2ds;\quad	y:=y^\phi=g^0(\int_0^\cdot\phi	ds)\big\}.
			\end{align}
			The main result is given as follows.
\begin{theorem}\label{Main-Thm-LDP}
	Assume $H_1$-$H_6$ hold.	 Let $\{u^\delta\}_{\delta>0}$ be the	unique	 solution of	\eqref{I-LDP}. Then 
	\begin{enumerate}
	\item		  $\{u^\delta\}_\delta$ satisfies LDP on $C([0,T];H)$,	as $\delta \to 0$, with the rate function $I$ given by  \eqref{rate-fctn}. 
	\item		If moreover $H_{7}$ holds. Then $\{u^\delta\}_\delta$ satisfies LDP on $C([0,T];H)\cap L^p(0,T;V)$,	as $\delta \to 0$, with the rate function $I$ given by  \eqref{rate-fctn}. 
	\end{enumerate}
\end{theorem}
\begin{proof}
Theorem	\ref{Main-Thm-LDP}	is	a	consequence	of	Lemma	\ref{LDP-cd2}	and	Lemma	\ref{LDP-cd1},	see	Section	\ref{section-LDP-proof}.
\end{proof}
\subsubsection*{Example}
For	$1<p<\infty$,	consider	the	p-Laplace	operator		given	by
$$	B(u)=-\text{div}[\vert	\nabla	u\vert^{p-2}\nabla	u],\quad	u	\in	W^{1,p}_0(D)=V.$$
The	unique	solution	$(u_\delta)_{\delta>0}$	of	\eqref{I-LDP},	with	$A=B$	and	$f,\psi,u_0$	and	$G$	satisfying	$H_3$-$H_6$,	satisfies	LDP	in	$C([0,T];H)\cap L^p(0,T;V)$	when	$p\geq2$	and	satisfies		LDP			in		$C([0,T];H)$		when	$\max(1,\frac{2d}{d+2})<p<2$.
\section{Well-posedness	of	skeleton Equations	\eqref{Skeleton}}\label{Sec-Skeleton}	Let	us start with the following result,	which ensures  the uniqueness  of the solution   to \eqref{Skeleton}.
\begin{lemma}\label{uniq-Ske}
Let $(y_1,R_1)$ and $(y_2,R_2)$ be  two  solutions to  \eqref{Skeleton}  in the sense of Definition \ref{Def-Skelton} associated  with  $(f_1,\phi_1)$	and	$(f_2,\phi_2)$. Then, there exists a positive constant $C>0$ such that 
	$$ \displaystyle\sup_{t\in[0,T]}\Vert (y_1-y_2)(t)\Vert_H^2 \leq  C\Vert f_1-f_2\Vert_{L^{p^\prime}(0,T;V^\prime)}\Vert y_1-y_2\Vert_{L^p(0,T;V)}+\int_0^T\Vert	\phi_1-\phi_2\Vert_{H_0}^2ds .$$
	Moreover,	if	$H_7$	holds	then
	\begin{align*}
\displaystyle\sup_{t\in[0,T]}\Vert (y_1-y_2)(t)\Vert_H^2&+\int_0^T	\Vert (y_1-y_2)(s)\Vert_V^pds \\&\quad\leq  C\Vert f_1-f_2\Vert_{L^{p^\prime}(0,T;V^\prime)}\Vert y_1-y_2\Vert_{L^p(0,T;V)}+C\int_0^T\Vert	\phi_1-\phi_2\Vert_{H_0}^2ds ,
	\end{align*}
		where
		$C:=C(M,\phi,G)$	and	$\phi_1,\phi_2\in	L^2([0,T];H_0)$.
\end{lemma}
\subsubsection*{Proof	of	Lemma	\ref{uniq-Ske}} Let $t\in [0,T]$ and   $(y_1,R_1)$ and $(y_2,R_2)$  are  two  solutions to  \eqref{Skeleton}  in the sense of Definition \ref{Def-Skelton}	with
	$(f_1,\phi_1)$	and	$(f_2,\phi_2)$.
By using $y_1-y_2$ as test function	in	the	equations	satisfied	by	$y_1-y_2$ and integrating in time from $0$ to $t$, one gets 
	\begin{align*}
	\dfrac{1}{2}\Vert (y_1-y_2)(t)\Vert_H^2&+\displaystyle\int_0^t\langle A(y_1,\cdot)-A(y_2,\cdot),y_1-y_2\rangle ds-\int_0^t(G(\cdot, y_1)\phi_1-G(\cdot, y_2)\phi_2, y_1-y_2)ds\\&+\displaystyle\int_0^t\langle  R_1-R_2, y_1-y_2 \rangle ds
	+ \displaystyle\int_0^t\langle  f_1-f_2, y_1-y_2 \rangle ds	\end{align*}
	
	\begin{itemize}
		
		\item[$\bullet$]  Since $y_1,y_2$	are	solutions	to	\eqref{Skeleton}	in	the	sense	of	Definition	\ref{Def-Skelton}, it yields  
		$$\langle  R_1-R_2, y_1-y_2 \rangle =\langle  R_1, y_1-y_2 \rangle +\langle R_2, y_2-y_1\rangle  \geq 0	\quad	\text{a.e. in 	}[0,T].$$
		Therefore 
		$\displaystyle\int_0^t\langle  R_1-R_2, y_1-y_2 \rangle ds=\int_0^t(\langle  R_1, y_1-y_2 \rangle +\langle R_2, y_2-y_1\rangle) ds  \geq 0 .$
		
		\item[$\bullet$] By using $H_3$, we get
		\begin{align*}
		&\int_0^t(G(\cdot, y_1)\phi_1-G(\cdot, y_2)\phi_2, y_1-y_2)ds\\
		 &\leq\int_0^t \Vert G(\cdot, y_1)\Vert_{L_2(H_0,H)} \Vert \phi_1-\phi_2\Vert_{H_0} \Vert y_1-y_2\Vert_Hds\\&\qquad+ \int_0^t \Vert G(\cdot, y_1)-G(\cdot, y_2)\Vert_{L_2(H_0,H)} \Vert \phi_2\Vert_{H_0} \Vert y_1-y_2\Vert_Hds\\
&\leq \displaystyle\int_0^t \Vert \phi_1-\phi_2\Vert_{H_0}^2 ds	+\displaystyle\int_0^t (M+\Vert \phi_2\Vert_{H_0}^2+\Vert G(\cdot, y_1)\Vert_{L_2(H_0,H)}^2)\Vert y_1-y_2\Vert_H^2 ds. 
			\end{align*}
		\item[$\bullet$]  
						From	$H_2$, one	has $\displaystyle\int_0^t\langle A(y_1,\cdot)-A(y_2,\cdot),y_1-y_2\rangle ds \geq -\lambda_T\int_0^t\Vert y_1-y_2\Vert_H^2 ds.$	Moreover,
	 if	 H$_7$	holds, we  have
		$$\displaystyle\int_0^t\langle A(y_1,\cdot)-A(y_2,\cdot),y_1-y_2\rangle ds \geq -\lambda_T\int_0^t\Vert y_1-y_2\Vert_H^2 ds+\int_0^t	\bar{\alpha}\Vert y_1-y_2\Vert^pds.$$
					\item Holder	inequality	gives	$$\displaystyle\sup_{t\in[0,T]}\vert \int_0^t\langle f_1-f_2, y_1-y_2 \rangle ds\vert  \leq \Vert f_1-f_2\Vert_{L^{p^\prime}(0,T;V^\prime)}\Vert y_1-y_2\Vert_{L^p(0,T;V)}.$$
	\end{itemize}

	Therefore,  Gr\"onwall's lemma ensures the existence of $C>0$ such that 
	\begin{align*}
\displaystyle\sup_{t\in[0,T]}\Vert (y_1-y_2)(t)\Vert_H^2&+\int_0^T	\Vert y_1-y_2\Vert^pds	 \\
&\quad\leq C\big(\Vert f_1-f_2\Vert_{L^{p^\prime}(0,T;V^\prime)}\Vert y_1-y_2\Vert_{L^p(0,T;V)}+\int_0^T \Vert \phi_1-\phi_2\Vert_{H_0}^2 ds\big),
	\end{align*} 
	where
	$C:=C(M,\phi,G)<\infty$.\\
	
		We	will	prove	Proposition	\ref{Prop-Skeleton}	in	two	steps.	First,	Subsection	\ref{subs-ske-regul}	is	devoted	to	the	proof	of	Proposition	\ref{Prop-Skeleton}	with		regular		data.	Next,	we	will	get	the	result	in	the	general	setting	by	using	some	approximations	in	Subsection	\ref{subs-sket-gle}.
		\subsection{Well-posedeness	of	\eqref{Skeleton}	with	regular	data}\label{subs-ske-regul}
		First,		consider	the	following	assumptions:	\begin{enumerate}
		\item[$H^*$:]	Assume that $\phi$ in $L^\infty(0,T;H_0)$ and denotes by $C_\phi=\Vert \phi\Vert_{L^\infty(0,T;H_0)}$.
		
		\item[$H^{**}$:] Assume  that $h^-$ is a  non negative element of $L^{\tilde{q}^\prime}(0,T; L^{\tilde{q}^\prime}(D))$	where	$\tilde{q}=\min(2,p)$.
	\end{enumerate}		The	proof	of	following	Theorem	\ref{THm-approx-1}	will	result	from	Subsections	\ref{subs-pen}-\ref{subsection-LS}.
\begin{theorem}\label{THm-approx-1}
	Under Assumptions (H$_1$)-(H$_6$) and assuming moreover  that $ h^-\in L^{\tilde{q}^\prime}(0,T; L^{\tilde{q}^\prime}(D)),	$	and	$H^*$	holds.  There exists  a unique solution $(y,\rho)\in L^p(0,T;V) \times L^{\tilde{q}^\prime}(0,T;L^{\tilde{q}^\prime}( D))$
	such that: 
	\begin{itemize}
		\item[i.] $y \in  C([0,T];H)$,	$y\geq\psi$, $ y(0)=u_0$	and	 $\rho\leq 0$.
		\item[ii.]	$\rho(y-\psi)=0$		a.e.	in	$[0,T]\times	D $	and	
		\begin{align*}
	\forall v \in L^p(0,T;V);	v\geq	\psi: 	\quad \langle \rho,y-v\rangle  \geq 0 \text{	a.e. in	} [0,T]. 		\end{align*}
		\item[iii.] For all $t\in[0,T]$,  
		$$y(t)+\displaystyle\int_0^t\rho	ds+\displaystyle\int_0^t A(y,\cdot)ds=u_0+\displaystyle\int_0^t G(y,\cdot)\phi	ds+\displaystyle\int_0^t fds.$$ 
		\item[iv.]The following  Lewy-Stampacchia inequality holds: $$0\leq  \dfrac{dy}{dt}+ A(y,\cdot)- G(y,\cdot)\phi-f \leq h^-=  \left(f - \partial_t  \psi  - A(\psi,\cdot)\right) ^-.$$
			\end{itemize}
\end{theorem}

\subsubsection{Penalization}\label{subs-pen}

Let $\epsilon>0$ and consider the following  approximation: 	
\begin{align}\label{Sket-penalization}
\left\{ \begin{array}{l}
\dfrac{dy_\epsilon}{dt}+(A(y_\epsilon,\cdot)-\widetilde{G}(y_\epsilon,\cdot)\phi-\dfrac{1}{\epsilon}[(y_\epsilon-\psi)^-]^{\tilde{q}-1}=f\\[1.2ex]
y_\epsilon(0)=u_0,
\end{array}
\right.
\end{align}
where 	$\widetilde{G}(y_\epsilon,\cdot)=G(\max(y_\epsilon,\psi),\cdot)$\footnote{	The proposed perturbation	of	$\widetilde{G}$	  makes formally the	term	coming	from	$G$	vanishes	on the free-set where the constraint is violated,	which	play	a	crucial	role	to	estimate	the	reflection	due	to	the	obstacle.}. 
Denote	by $\bar{A}(y_\epsilon,\cdot)=A(y_\epsilon,\cdot)-\widetilde{G}(\cdot,y_\epsilon)\phi-\dfrac{1}{\epsilon}[(y_\epsilon-\psi)^-]^{\tilde{q}-1}-f$. 
\begin{enumerate}
	\item[(i)] 	Note	that	$
	\widetilde{G}(\cdot,y_\epsilon)\phi: [0,T]\times H \to H$ satisfies
	\begin{align*}
	\Vert \widetilde{G}(\cdot,y_\epsilon)\phi\Vert_H^2 \leq \Vert G(\cdot,\max( y_\epsilon,\psi))\Vert_{L_2(H_0,H)}^2 \Vert \phi\Vert_{H_0}^2\leq	L	C_\phi^2 (1+\Vert	y_\epsilon\Vert_{H}^2+\Vert	\psi\Vert_{H}^2).\end{align*}
		Thus,		by the	properties	of	the	penalization	term,		$A$	and	$f$	(see		\cite[Section	3.1]{YV22}), we	get	that	$\bar{A}$ is an operator defined on $V\times[0,T]$ with values in $V^\prime$. 
	\item[(ii)] (Local	monotonicity) Let $y_1,y_2 \in V$,	by using $H_{3,1}$ we get
	\begin{align*}
	([\widetilde{G}(\cdot,y_1)-\widetilde{G}(\cdot,y_2)]\phi,y_1-y_2) \leq M\Vert \phi\Vert_{H_0}\Vert y_1-y_2\Vert_H^2\leq MC_\phi\Vert y_1-y_2\Vert_H^2.
	\end{align*}
	   Thanks to $H_{2,2}$,  $H_{3,1}$	and		$x \mapsto -x^-$ is non-decreasing, there exists $C_1\in \mathbb{R}$ such that 
	\begin{align*}
	\langle \bar{A}(y_1,\cdot)-\bar{A}(y_2,\cdot),y_1-y_2\rangle \geq 	C_1\Vert y_1-y_2\Vert_H^2
	\end{align*}
	\item[(iii)]	 The structure of the penalization operator  and $H_{3,1}$ yield the hemi-continuity of $\bar{A}$.
	\item[(iv)]  (Coercivity): Note that for any $\delta>0$, there exists   $C_{\delta,\epsilon} >0$ such that: $\forall v \in V,$ 
	\begin{align*}
	\langle f,v\rangle \leq & C_\delta \Vert f\Vert_{V^\prime}^{p^\prime}+ \delta\Vert v\Vert_{V}^p \quad,
	\\
	\langle-\dfrac{1}{\epsilon}[(v-\psi)^-]^{\tilde{q}-1},v\rangle \geq &  \langle-\dfrac{1}{\epsilon}[(v-\psi)^-]^{\tilde{q}-1},\psi\rangle
	\\ \geq & -\delta \Vert v\Vert _{L^{\tilde{q}}(D)}^{\tilde{q}}- C_{\delta,\epsilon} \Vert \psi\Vert _{L^{\tilde{q}}(D)}^{\tilde{q}}\geq  -\delta C \Vert v\Vert _{V}^{p}- C_{\delta,\epsilon} \Vert \psi\Vert _{L^{\tilde{q}}(D)}^{\tilde{q}}-C_{\delta,\epsilon},\\
	(\widetilde{G}(\cdot,v)\phi,v) \leq&  \Vert \widetilde{G}(\cdot,v)\Vert_{L_2(H_0,H)}\Vert \phi\Vert_{H_0}\Vert v\Vert_H\leq  \Vert \widetilde{G}(\cdot,v)\Vert_{L_2(H_0,H)}^2+
C_\phi^2	\Vert v\Vert_H^2\\
	\leq& (L+C_\phi^2)\Vert v\Vert_H^2+L (1+\Vert	\psi\Vert_{H}^2),
	\end{align*}
	where $C$ is related  to the continuous embedding of $V$ in $L^{\tilde{q}}(D)$.		Denote by 
	$\tilde{l}_1=L (1+\Vert	\psi\Vert_{H}^2)+l_1+C_\delta \Vert f\Vert_{V^\prime}^{p^\prime}+C_{\delta,\epsilon} \Vert \psi\Vert _{L^{\tilde{q}}(D)}^{\tilde{q}}$, it is  $L^\infty([0,T])$  thanks to the assumptions on $f$ and $\psi$, depending only on the data. Therefore, by a convenient choice of $\delta$, $\bar{A}$ satisfies $H_{2,1}$ by considering $\tilde{l}_1$ instead  of $l_1$.
	\item[(v)] (Growth):  Let	$ v \in V,$
	\begin{align*}
	\| -\dfrac{1}{\epsilon}[(v-\psi)^-]^{\tilde{q}-1} \|_{L^{\tilde q^\prime}(D)} =\dfrac{1}{\epsilon}\|(v-\psi)^- \|_{L^{\tilde q}(D)}^{\tilde{q}-1} &\leq C_\epsilon\left(\|v\|_{L^{\tilde q}(D)}^{\tilde{q}-1}+\|\psi\|_{L^{\tilde q}(D)}^{\tilde{q}-1} \right) \\
	&\leq C_\epsilon\left(\|v\|_{L^{\tilde q}(D)}^{p-1}+\|\psi\|_{L^{\tilde q}(D)}^{p-1}\right) +C_p
	\end{align*} 
	since $\tilde q <p$ may be possible. Now, since the embeddings of $L^{\tilde q^\prime}(D)$ in $V^\prime$ and of  $V$  in $L^{\tilde{q}}(D)$ are continuous,  one	has
$
	\| -\dfrac{1}{\epsilon}[(v-\psi)^-]^{\tilde{q}-1} \|_{V^\prime}  \leq C_\epsilon\left(\|v\|_{V}^{p-1}+\|\psi\|_{V}^{p-1}\right) +C_p
$.
	Moreover,		
$
		\Vert \widetilde{G}(\cdot,v)\phi\Vert_H \leq  \Vert \widetilde{G}(\cdot,v)\Vert_{L_2(H_0,H)}\Vert \phi\Vert_{H_0}\leq C_\phi^2+L\Vert v\Vert_H^2+L(1+\Vert \psi\Vert_H^2).
$
	Since	 the embeddings of $H$ in $V^\prime$	is	continuous,	we	get	
	\begin{align*}
		\Vert\bar{A}(\cdot,v)\Vert_{V^\prime}&\leq	C_\phi^2+L\Vert v\Vert_H^2+L(1+\Vert \psi\Vert_H^2)+C_\epsilon\left(\|v\|_{V}^{p-1}+\|\psi\|_{V}^{p-1}\right) +C_p+g+\bar{K}\Vert	v\Vert_{V}^{p-1}\\
		&\leq	K_\epsilon\Vert	v\Vert_{V}^{p-1}+L\Vert v\Vert_H^2+\tilde{g},
	\end{align*}
	where	$\tilde{g}=C_\epsilon\|\psi\|_{V}^{p-1}++L(1+\Vert \psi\Vert_H^2)+g+C_\phi^2+C_p	\in	L^\infty([0,T])$.	Therefore	$\bar{A}$	satisfies	$H_{2,3}$	with	$\tilde{g}$	instead	of	$g$.	
\end{enumerate}
By	using	 \cite[Thm. 5.1.3]{RoK15}, for all $\epsilon>0$, there exists  a unique 	solution $u_\epsilon \in L^p(0,T;V)\cap	C([0,T];H)$   to	\eqref{Sket-penalization}	and satisfying:  
\begin{align*}
y_\epsilon(t)+\int_0^t(A(y_\epsilon,\cdot)-\widetilde{G}(y_\epsilon,\cdot)\phi-\dfrac{1}{\epsilon}[(y_\epsilon-\psi)^-]^{\tilde{q}-1})ds=u_0+\int_0^tfds\quad	\text{	in	}	V^\prime	\text{ for all }	t\in[0,T].
\end{align*}

\subsubsection{Uniform estimates}\label{sub-estim-pen}
\begin{lemma}\label{lem1}
	\begin{enumerate}
		\item $(y_\epsilon)_{\epsilon>0}$ is bounded in  $ L^p(0,T;V)\cap C([0,T];H)$.
		\item $(A(y_\epsilon,\cdot))_{\epsilon>0}$ is bounded in $L^{p^\prime}(0,T;V^\prime).$ 
		\item  $ (\dfrac{(y_\epsilon-\psi)^-}{\epsilon^{\frac{1}{\tilde{q}}}})_\epsilon$	is	bounded	in	$L^{\tilde{q}}([0,T]\times	D).$
	\end{enumerate}
\end{lemma}
\begin{proof}
	Let $\epsilon >0$ and note	that
$\dfrac{d(y_\epsilon-\psi)}{dt}+\Big(A(y_\epsilon,\cdot)-\widetilde{G}(\cdot,y_\epsilon)\phi-\dfrac{1}{\epsilon}[(y_\epsilon-\psi)^-]^{\tilde{q}-1})
	=	[ f- \dfrac{d\psi}{dt}]. $

	By using $y_\epsilon-\psi$ as a test function and integrating in time from $0$ to $t$, we get
	\begin{align*}
	&\|(y_\epsilon-\psi)(t)\|_{H}^2+2\displaystyle\int_0^t \langle A(y_\epsilon,\cdot),y_\epsilon-\psi\rangle ds -2\int_0^t(\widetilde{G}(\cdot,y_\epsilon)\phi,y_\epsilon-\psi)ds
	\\ &-2\int_0^t \int_D\dfrac{1}{\epsilon}[(u_\epsilon-\psi)^-]^{\tilde{q}-1}(u_\epsilon-\psi)dxds = \|u_0-\psi(0)\|^2_{H}
	+
	2\int_0^t \langle  f- \dfrac{d\psi}{ds}  , y_\epsilon-\psi \rangle ds.
		\end{align*}
	Note that 
$
	-2\displaystyle\int_0^t \int_D\dfrac{1}{\epsilon}[(y_\epsilon-\psi)^-]^{\tilde{q}-1}(y_\epsilon-\psi)dxds
	 = \dfrac{2}{\epsilon}\int_0^t \!\!\!\int_D[(y_\epsilon-\psi)^-]^{\tilde{q}}dxds.
$
	By using $H_{2,3}$	and	$H_3$,		 we get
	\begin{align*}
		(\widetilde{G}(\cdot,y_\epsilon)\phi,y_\epsilon-\psi) &\leq \Vert \widetilde{G}(\cdot,y_\epsilon)\Vert_{L_2(H_0,H)}\Vert \phi\Vert_{H_0}\Vert y_\epsilon-\psi\Vert_H	
		\leq	L\Vert y_\epsilon\Vert_H^2+C_\phi^2\Vert y_\epsilon-\psi\Vert_H^2+C(\psi);
	\end{align*}
	\begin{align*}
	\langle A(y_\epsilon,\cdot),y_\epsilon-\psi\rangle  \geq& \alpha \Vert y_\epsilon\Vert_V^p - \lambda \Vert y_\epsilon\Vert_H^2-l_1 -  \langle A(y_\epsilon,\cdot),\psi\rangle
	\\
	\geq& \alpha \Vert y_\epsilon\Vert_V^p - \lambda \Vert y_\epsilon\Vert_H^2-l_1  - \bar K\Vert y_\epsilon\Vert_V^{p-1}\Vert \psi\Vert_V-g\Vert \psi\Vert_V
	\\
	\geq& \frac{\alpha}{2} \Vert y_\epsilon\Vert_V^p - \lambda \Vert y_\epsilon\Vert_H^2-l_1  - C(\psi),
	\end{align*}
	where $C(\psi) \in L^1([0,T])$. Thus, for any positive $\gamma$, Young's inequality yields the existence of a positive constant $C_\gamma$ that may change form line to line, such that 
	\begin{align*}
	&\|(y_\epsilon-\psi)(t)\|_{H}^2+2\displaystyle \int_0^t \frac{\alpha}{2} \Vert y_\epsilon(s) \Vert_V^p ds +\dfrac{2}{\epsilon}\int_0^t \!\!\!\int_D[(y_\epsilon-\psi)^-]^{\tilde{q}}dxds
	\leq  
	(\lambda+L) \int_0^t \Vert y_\epsilon(s)\Vert_H^2  ds\\&+ \|l_1+C(\psi)\|_{L^1([0,T])} + C_\gamma(f, \dfrac{d\psi}{dt} ) 
	+ \gamma \int_0^t\|(y_\epsilon-\psi)(s) \|^p_V ds+C_\phi^2\int_0^t\Vert (y_\epsilon-\psi)(s)\Vert_H^2ds+Lt 
	\\ &\qquad\leq  
	C \int_0^t \Vert (y_\epsilon - \psi)(s)\Vert_H^2  ds+ \frac{\alpha}{2}\int_0^t\|y_\epsilon(s)\|^p_V ds+ C,
	\end{align*}
	for a suitable choice of $\gamma$. 
	Then, the first part	and	the third one	 of the lemma is proved by  Gronwall's lemma,  the second   by adding H$_{2,3}$.
\end{proof}
Note	that					Lemma	\ref{lem1}$(3)$	is	not	sufficient	to	pass	to	the	limit	in	the	penalization	term.	Thus,	we	will	prove	the	next	result,	where	$H_4$,	$H_5$	and	the	perturbation	of	$G$	will	play	a	crucial	role.

\begin{lemma}\label{lem2-LDP} 
	$(\dfrac{1}{\epsilon}[(y_\epsilon-\psi)^-]^{\tilde{q}-1})_{\epsilon>0}$ is bounded in $L^{\tilde{q}^\prime}([0,T]\times D).$
\end{lemma}
\begin{proof}
From \eqref{Sket-penalization}, we have
	\begin{align*}
	&\dfrac{d(y_\epsilon-\psi)}{dt}+\Big(A(y_\epsilon,\cdot)-A(\psi,\cdot)-\widetilde{G}(\cdot,y_\epsilon)\phi-\dfrac{1}{\epsilon}[(y_\epsilon-\psi)^-]^{\tilde{q}-1})
	=[ f- \dfrac{d\psi}{dt}-A(\psi,\cdot)] 
	\end{align*}
		By using the admissible test function $-(y_\epsilon-\psi)^-$ in \eqref{Sket-penalization}, integrating in time from $0$ to $t$	(see		\cite[Corollary	4.5]{GMYV})	and	using	that	$u_0\geq		\psi(0)$, we obtain
	\begin{align*}
	&\|(y_\epsilon-\psi)^-(t) \|_{H}^2+2\displaystyle\int_0^t \langle A(y_\epsilon,\cdot)-A(\psi,\cdot),-(y_\epsilon-\psi)^-\rangle ds 
	+2\int_0^t \int_D\dfrac{1}{\epsilon}[(y_\epsilon-\psi)^-]^{\tilde{q}}dxds \\ &= 2\int_0^t(\widetilde{G}(\cdot,y_\epsilon)\phi,-(y_\epsilon-\psi)^-)ds
		+2\int_0^t \langle  f- \dfrac{d\psi}{ds}-A(\psi,\cdot)  , -(y_\epsilon-\psi)^- \rangle ds\\
		&\leq  2\int_0^t \langle  -h^-    , -(y_\epsilon-\psi)^- \rangle ds.		
	\end{align*}
	Note	that	$(\widetilde{G}(\cdot,y_\epsilon)\phi,-(y_\epsilon-\psi)^-)=(G(\cdot,\psi)\phi,-(y_\epsilon-\psi)^-)=0$	a.e.	in	$[0,T]$,	since	$\widetilde{G}(\cdot,y_\epsilon)=G(\cdot,\psi)$	on	the	set	$\{y_\epsilon\leq	\psi\}$	.

		On	the	other	hand,	$H_{2,2}$	ensures
	$\langle A(y_\epsilon,\cdot)-A(\psi,\cdot),-2(y_\epsilon-\psi)^-\rangle\geq -2\lambda_T \Vert (y_\epsilon-\psi)^-\Vert_H^2,$ a.e. $t\in [0,T]$,	since  this last term is equal to  $2\langle A(\psi,\cdot)-A(y_\epsilon,\cdot),(\psi-y_\epsilon)^+\rangle \geq -2\lambda_T \Vert (\psi-y_\epsilon)^+\Vert_H^2$.
	By	using	$H_5$,	we	get
		\begin{align}
		\Vert(y_\epsilon-\psi)^-(t)\Vert_{L^2(D)}^2&+\dfrac{2}{\epsilon} \displaystyle\int_0^t\Vert (y_\epsilon-\psi)^-(s)\Vert_{L^{\tilde{q}}(D)}^{\tilde{q}}ds \nonumber \\  &\leq 2\displaystyle\int_0^t\langle h^-(s), (y_\epsilon-\psi)^-(s)\rangle ds+2\lambda_T \int_0^t\Vert (y_\epsilon-\psi)^-(s)\Vert_H^2ds.\label{Pen}
		\end{align}
			We	are	in	position	to	use		similar		arguments	
			to	the	one	used	in	the	last	part	of			the	proof	of		\cite[Lemma	3]{YV22}	to	conclude.\end{proof}

As a consequence,	the following lemma holds.
\begin{lemma}\label{lem3} 
	$(y_\epsilon)_{\epsilon>0}$ is  a Cauchy sequence in the space $C([0,T];H)$.
\end{lemma}

\begin{proof}
	Let $1>\epsilon\geq \eta >0$ and consider  $y_\epsilon-y_\eta$, which satisfies the following equation
	\begin{align*}
	y_\epsilon(t)-y_\eta(t)+\displaystyle\int_0^t (A(y_\epsilon,\cdot)-A(y_\eta,\cdot))&+ (-\dfrac{1}{\epsilon}[(y_\epsilon-\psi)^-]^{\tilde{q}-1}+\dfrac{1}{\eta}[(y_\eta-\psi)^-]^{\tilde{q}-1})ds\\
	& =\displaystyle\int_0^t (\widetilde{G}(y_\epsilon,\cdot)\phi-\widetilde{G}(y_\eta,\cdot)\phi)ds. 
	\end{align*}
	By	using	$y_\epsilon-y_\eta$	as	test	function	and	integrate	from	$0$	to	$t$, one gets for any $t\in [0,T]$
	\begin{align*}
	&\dfrac{1}{2}\Vert (y_\epsilon-y_\eta)(t)\Vert_H^2+\displaystyle\int_0^t \langle A(y_\epsilon,\cdot)-A(y_\eta,\cdot),y_\epsilon-y_\eta\rangle ds\\
	&\quad+\displaystyle\int_0^t\langle  -\dfrac{1}{\epsilon}[(y_\epsilon-\psi)^-]^{\tilde{q}-1}+\dfrac{1}{\eta}[(y_\eta-\psi)^-]^{\tilde{q}-1}, y_\epsilon-y_\eta \rangle ds= \displaystyle\int_0^t\displaystyle\int_0^t (\widetilde{G}(y_\epsilon,\cdot)\phi-\widetilde{G}(y_\eta,\cdot)\phi, y_\epsilon-y_\eta) ds	\end{align*}
	We argue as in the proof of	Lemma \ref{uniq-Ske} with $f_1=f_2$ and note that we need only to discuss  the penalization  term.	By	using the monotonicity of  the penalization  operator,	arguments	already	detailed	in	the	proof	of	\cite[Lemma	4]{YV22} lead	to
		$$ \displaystyle\sup_{t\in[0,T]}\Vert (y_\epsilon-y_\eta)(t)\Vert_H^2 \leq C(\epsilon+\epsilon^{\frac{1}{p-1}})+C\displaystyle\int_0^T\sup_{\tau\in[0,s]}\Vert (y_\epsilon-y_\eta)(\tau)\Vert_H^2 ds$$
	and  Gr\"onwall's  lemma ensures that $(y_\epsilon)_{\epsilon>0}$ is a Cauchy sequence in the space $C([0,T];H)$. 
\end{proof}

\subsubsection{Existence	of	solution}\label{limit-passage-pen}
\begin{lemma}\label{Lem2.4}
	There exist $y \in L^p(0,T;V)\cap  C([0,T];H) $ and $ (\rho,\chi ) \in L^{\tilde{q}^\prime}(0,T;L^{\tilde{q}^\prime}(D))\times L^{p^\prime}(0,T;V^\prime)$ such that the following convergences hold, up to sub-sequences denoted by the same way,
	\begin{align}
	y_\epsilon &\rightharpoonup y \quad \text{in} \quad  L^p(0,T;V),\label{3-8-Sk}\\
	y_\epsilon &\rightarrow y  \quad \text{in} \quad  C([0,T];H),\label{4-Sk}\\ 
	A(y_\epsilon,\cdot) &\rightharpoonup \chi  \quad \text{in} \quad L^{p^\prime}(0,T;V^\prime),\label{5-Sk}\\
	-\dfrac{1}{\epsilon}[(y_\epsilon-\psi)^-]^{\tilde{q}-1} &\rightharpoonup \rho,\quad \rho \leq 0 \quad \text{in} \quad L^{\tilde{q}^\prime}([0,T]\times D).\label{7-Sk}
	\end{align}
\end{lemma}
\begin{proof}
	By  compactness  with respect to the weak topology in the spaces $ L^p(0,T;V)$, $L^{p^\prime}(0,T;V^\prime)$ and $L^{\tilde{q}^\prime}([0,T]\times D)$, there exist $y \in L^p(0,T;V)$, $\chi \in L^{p^\prime}(0,T;V^\prime)$ and $\rho \in L^{\tilde{q}^\prime}([0,T]\times D) $ such that $\eqref{3-8-Sk}$, $\eqref{5-Sk}$ and \eqref{7-Sk}  hold (for  sub-sequences). 
	Thanks to  Lemma \ref{lem3}, we get   the strong convergence of $y_\epsilon$ to $y$ in $ C([0,T];H) \hookrightarrow L^2([0,T]\times D)$.
	Moreover, 
			   $\rho \leq 0$ since the set of non positive functions of $L^{\tilde{q}^\prime}([0,T]\times D)$ is a closed convex subset of $L^{\tilde{q}^\prime}([0,T]\times D)$.
\end{proof}
Concerning	the	initial condition and constraint,	we	get
	\begin{itemize}
		\item Lemma	\ref{lem3}	ensures	that  $y_\epsilon(0)=u_0$ converges to $ y(0)$ in $ H$ and $y(0)=u_0$ in $H$.
		\item Thanks   to Lemma \ref{lem2-LDP}	and	Lemma	\ref{lem3}, we deduce that  $(y_\epsilon-\psi)^-\rightarrow (y-\psi)^-=0$ in $L^{\tilde{q}}([0,T]\times D)$ and $y\geq \psi$	a.e.
	\end{itemize}

\begin{lemma}\label{lem6} 
	$\widetilde{G}(\cdot, y_\epsilon)\phi$	converges	to	$\widetilde{G}(\cdot, y)\phi=G(\cdot, y)\phi$ in $ L^2(0,T;H)$, as $\epsilon\rightarrow 0$.
\end{lemma}
\begin{proof}
We	have
	\begin{align*}
\int_0^T	\Vert	\widetilde{G}(\cdot, y_\epsilon)\phi-\widetilde{G}(\cdot, y)\phi\Vert_H^2ds&\leq	\int_0^T\Vert	\widetilde{G}(\cdot, y_\epsilon)-\widetilde{G}(\cdot, y)\Vert_{L_2(H_0,H)}^2\Vert\phi\Vert_{H_0}^2dt	\\
&\leq	MC_\phi^2\int_0^T\Vert	y_\epsilon-y\Vert_{H}^2dt\to	0.
	\end{align*}
$\widetilde{G}(\cdot, y)\phi=G(\cdot, y)\phi$	is	a	consequence	of	$y\geq	\psi	$.

\end{proof}

\begin{lemma}\label{lem7}
	$\rho(y-\psi)=0$  a.e. in $[0,T]\times	D$ and for	any	$ v\in L^p(0,T;V);	v\geq	\psi$, $\rho(y-v) \geq 0$ a.e. in 	$[0,T]\times	D.$
\end{lemma}
\begin{proof}
	On one hand, by Lemma \ref{lem2-LDP}, we have 
	$$  0\leq -\dfrac{1}{\epsilon} \displaystyle\int_0^t \langle [(y_\epsilon-\psi)^-]^{\tilde{q}-1}, y_\epsilon-\psi\rangle ds=\dfrac{1}{\epsilon} \displaystyle\int_0^t \Vert (y_\epsilon-\psi)^-(s)\Vert_{L^{\tilde{q}}}^{\tilde{q}}ds \leq C\epsilon^{\tilde{q}^\prime-1}\rightarrow 0.$$
	On the other hand, by Lemma \ref{Lem2.4}, we distinguish two cases:
	\begin{itemize}
		\item If $p\geq 2$ then $-\dfrac{1}{\epsilon}(y_\epsilon-\psi)^- \rightharpoonup  \rho$ in $L^2([0,T]\times D)$ and $ y_\epsilon-\psi \rightarrow y-\psi$ in $L^2([0,T]\times D)$   by  Lemma \ref{lem3}. Hence $ \displaystyle\int_0^T\hspace*{-0.3cm} \int_D\rho(y-\psi)dxdt=0$ and $\rho(y-\psi)=0$ a.e.	in	$[0,T]\times	D$,	since the integrand is always non-positive.
		\item If $ 2>p>1 $ then $-\dfrac{1}{\epsilon}[(y_\epsilon-\psi)^-]^{p-1} \rightharpoonup  \rho$ in $L^{p^\prime}([0,T]\times D)$ and $ y_\epsilon-\psi \rightarrow y-\psi$ in $L^p([0,T]\times D)$   by  Lemma \ref{lem3}  and the same conclusion holds.
	\end{itemize}
	One finishes  the proof by noticing that if $ v\in L^p(0,T;V);	v\geq	\psi$,  one has a.e. in $[0,T]\times	D$ that, 
	$$ \rho(y-v)  = \overbrace{\rho(y-\psi) }^{=0}+\overbrace{\rho(\psi-v) }^{\geq 0}\geq 0.$$
\end{proof}
Our aim now is to prove that $A(y,\cdot)=\chi$.  For any $t\in [0,T]$,	we have 
$$  y_\epsilon(t)-y(t)+\displaystyle\int_0^t[(A(y_\epsilon,\cdot) -\chi)+(-\dfrac{1}{\epsilon}[(y_\epsilon-\psi)^-]^{\tilde{q}-1}-\rho)]ds=\displaystyle\int_0^t[G(y_\epsilon,\cdot)-G(y,\cdot)]\phi	ds	\text{ in }		V^\prime.$$
By	using	
	 $y_\epsilon-y$	as	a		test	function	and	integrate	from	$0$	to	$t$,		we	obtain
\begin{align*}
\dfrac{1}{2}\Vert &(y_\epsilon-y)(t)\Vert_H^2+\overbrace{\displaystyle\int_0^t \langle A(y_\epsilon,\cdot) -\chi, y_\epsilon-y\rangle ds}^{I_1} +\overbrace{\displaystyle\int_0^t \langle -\dfrac{1}{\epsilon}[(y_\epsilon-\psi)^-]^{\tilde{q}-1}-\rho, y_\epsilon-y\rangle ds}^{I_2}\\
&=\overbrace{\displaystyle\int_0^t \langle (G(y_\epsilon,\cdot)-G(y,\cdot)) \phi, y_\epsilon-y\rangle ds}^{I_3} 
\end{align*}
Let  $v\in L^p(0,T;V)\cap C([0,T];H)$ and $t\in ]0,T]$.	Note the	following:
\begin{itemize}
	\item   $I_1=\displaystyle\int_0^t \langle A(y_\epsilon,\cdot), y_\epsilon\rangle ds-\displaystyle\int_0^t \langle A(y_\epsilon,\cdot), y\rangle ds-\displaystyle\int_0^t \langle \chi, y_\epsilon-y\rangle ds$ and  
	\begin{align*}
&	\displaystyle\int_0^t \langle A(y_\epsilon,\cdot), y_\epsilon\rangle ds= \displaystyle\int_0^t \langle A(y_\epsilon,\cdot)-A(v,\cdot), y_\epsilon-v\rangle ds+ \displaystyle\int_0^t \langle A(v,\cdot), y_\epsilon-v\rangle ds +\displaystyle\int_0^t \langle A(y_\epsilon,\cdot),v\rangle ds\\
&	(\lambda_T Id+A \text{ is T-monotone} )\quad \geq \displaystyle\int_0^t \langle A(v,\cdot), y_\epsilon-v\rangle ds +\displaystyle\int_0^t \langle A(y_\epsilon,\cdot),v\rangle ds-\lambda_T\int_0^t\Vert y_\epsilon -v \Vert_H^2ds.
	\end{align*}
	
	\item  We	have	$ \displaystyle\int_0^t \langle -\dfrac{1}{\epsilon}[(y_\epsilon-\psi)^-]^{\tilde{q}-1}, y_\epsilon-y\rangle ds \geq  \displaystyle\int_0^t \langle -\dfrac{1}{\epsilon}[(y_\epsilon-\psi)^-]^{\tilde{q}-1}, \psi-y\rangle ds.$
	\end{itemize}
 Thanks to $H_3$ we have $\vert	I_3\vert\leq MC_\phi^2\displaystyle\int_0^t\Vert y_\epsilon(s)-y(s)\Vert_H^2ds .$	Thus,	we	are	able	to	infer
\begin{align*}
\dfrac{1}{2}\Vert (y_\epsilon-y)(t)\Vert_H^2+ \displaystyle\int_0^t \langle A(v,\cdot), y_\epsilon-v\rangle ds +\displaystyle\int_0^t \langle A(y_\epsilon,\cdot),v-y\rangle ds-\displaystyle\int_0^t \langle \chi, y_\epsilon-y\rangle ds \\+\displaystyle\int_0^t \langle -\dfrac{1}{\epsilon}[(y_\epsilon-\psi)^-]^{\tilde{q}-1}, \psi-y\rangle ds-\displaystyle\int_0^t \langle \rho, y_\epsilon-y\rangle ds \leq (MC_\phi^2+\lambda_T)\displaystyle\int_0^t\Vert y_\epsilon(s)-y(s)\Vert_H^2ds.
\end{align*}
By setting $t=T$
and	 passing to the  limit as $\epsilon \rightarrow 0$,  thanks to Lemmas \ref{Lem2.4} and \ref{lem7},  we get  
$$ \displaystyle\int_0^T\langle A(v,\cdot)-\chi, y-v\rangle ds  \leq  \displaystyle\int_0^T\langle \rho, y-\psi\rangle ds=0. $$
We are now in a position  to use "Minty's trick" \cite[Lemma 2.13 p.35]{Roubicek} and  deduce that $A(y,\cdot)=\chi$.

\subsubsection{Lewy-Stampacchia's	inequality}\label{subsection-LS}
In	order	to	consider	more	general	setting,	one	needs	to	estimate		$\rho$	in	satisfactory	way.	Hence,		we	will	prove	Lewy-Stampacchia's	inequality	estimate,	which	gives		lower	and	upper	bounds	to	$\rho$	in	some	dual	sense,	where	the	dual	order	assumption	in	$H_5$	is	crucial.
First,	note	that	$\dfrac{dy}{dt}+A(y,\cdot)-G(y,\cdot)\phi-f=-\rho\geq	0.$
	Moreover,	we	have					
 \begin{align}\label{LS-regular}
 0\leq\dfrac{dy}{dt}+A(y,\cdot)-G(y,\cdot)\phi-f\leq	h^-	\text{ in }	L^{\tilde{q}^\prime}([0,T]\times	D)-sense.
  \end{align}
 	Indeed,	let $y$ be the unique  solution given	at	the	end	of		Subsection	\ref{limit-passage-pen}.	 Denote by $K_1$ the closed convex  set   $K_1=\{v\in L^p(0,T;V) , \quad v\leq y \quad \text{a.e. in }D\times [0,T]\}.$
  	We recall that $y$ satisfies  
 	$$(f+h^-)-\dfrac{dy}{dt}- A(y,\cdot)+ G(y,\cdot)\phi=h^-+\rho, \quad \rho \leq 0, \quad \rho \in L^{\tilde{q}^\prime}([0,T]\times D).$$ 
 	Consider the following auxiliary problem: $(z,\nu)\in L^p(0,T;V) \times L^{\tilde{q}^\prime}(0,T;L^{\tilde{q}^\prime}(D))$ such that  
 	\begin{align}\label{8}
 	\left\{ \begin{array}{l}
 	i.)\quad  z \in C([0,T];H),\quad z(0)=u_0 \quad \text{and } \quad  z \in K_1, 
 	\\[0.2cm]
 	ii.)\quad  \nu \geq 0, \quad \langle \nu, z-y\rangle=0 \text{ a.e. in }[0,T]\text{ and }\forall v\in K_1, \ \langle \nu, z-v\rangle\geq 0   \text{ a.e. in }[0,T].
 	\\[0.2cm] 
 	iii.) \quad \text{For any } t\in[0,T] :\\
 	z(t)+\displaystyle\int_0^t \nu ds+\displaystyle\int_0^t A(z,\cdot)ds=u_0+\displaystyle\int_0^t G(z,\cdot)\phi	ds+\displaystyle\int_0^t (f+h^-) ds.
 	\end{array}
 	\right.
 	\end{align}
 	
 	Note that the result of existence and uniqueness of the solution $(z,\nu)$	to	\eqref{8} can be proved,
 	 	  by cosmetic changes of what has been done in the		Subsections	\ref{subs-pen}-\ref{limit-passage-pen},  by passing to the limit in the following penalized problem: 
 	\begin{align*}
 	\left\{ \begin{array}{l}
 	z_\epsilon(t)+\displaystyle\int_0^t(A(z_\epsilon,\cdot)+\dfrac{1}{\epsilon}[(z_\epsilon-y)^+]^{\tilde{q}-1}-(f+h^-))ds=u_0+\displaystyle\int_0^t\widetilde{G}(z_\epsilon,\cdot)\phi	ds \\[1.2ex]
 	z_\epsilon(0)=u_0,
 	\end{array}
 	\right.
 	\end{align*}
 	where $\widetilde{G}(z_\epsilon,\cdot)=G(\min(z_\epsilon,y),\cdot).$
 	Moreover, $$  \dfrac{dz}{dt}+ A(z,\cdot)- G(z,\cdot)\phi-(f+h^-)=-\nu \leq 0\quad \text{ in} \quad  L^{\tilde{q}^\prime}([0,T]\times D)$$ and $z$ satisfies the following Lewy-Stampacchia inequality:
 	$$ \dfrac{dz}{dt}+ A(z,\cdot)- G(z,\cdot)\phi-f \leq h^- \quad \text{ in} \quad  L^{\tilde{q}^\prime}([0,T]\times D). $$
 	\\
 	Since	the	solution	of	\eqref{Skeleton}	in	Theorem	\ref{THm-approx-1}	is	unique,	the	proof	of		Lewy-Stampacchia's	inequality	\eqref{LS-regular}	follows	by	using	the	same	arguments	presented		in	\cite[Subsection	3.2]{YV22},	by	showing	that	$z=y$.
 
	\subsection{Well-posedeness	of	\eqref{Skeleton}	with general data}\label{subs-sket-gle}	We 	will	proceed in two steps.
\subsubsection{The	case	$\phi\in	L^2(0,T;H_0)$}\label{remove-phi}
Assume	only	in	this	part	that	$H^{**}$	holds.
Let $m\in\mathbb{N}$,	let $\phi_m\in L^\infty([0,T];H_0)$\footnote{	$(\phi_m)_m$	can	be	constructed		\textit{e.g.}	by	using	a	standard	cut-off	techniques.} such that $\phi_m \to \phi $ in $L^2([0,T];H_0)$.	Following	Theorem	\ref{THm-approx-1},	
 there exists a unique $(y_m,\rho_m)\in L^p(0,T;V) \times L^{\tilde{q}^\prime}(0,T;L^{\tilde{q}^\prime}(D))$  satisfying: 
\begin{itemize}
	\item  $y_m\in  C([0,T],H)$,	$y_m\geq		\psi$ and $\rho_m \leq 0$.
	\item For any $t\in[0,T]$: $ y_m(t)+\displaystyle\int_0^t(A(y_m,\cdot)+\rho_m)ds=u_0+\displaystyle\int_0^t G(y_m,\cdot)\phi_m+\int_0^tfds$	in	$V^\prime$.
	\item $\langle \rho_m, y_m-\psi\rangle=0$ a.e.	in	$(0,T)$ and $\forall	v\in L^p(0,T;V);	v\geq	\psi$,  $\langle \rho_m, y_m-v\rangle\geq 0$ a.e. in 	$(0,T)$.
	\item	The following  Lewy-Stampacchia inequality holds: $0\leq  -\rho_m \leq h^-.$
\end{itemize}
We	will	to	pass	to	the	limit	as	$m\to	+\infty$.	The	first	step	is	to	obtain	a	uniform	estimate	independent	of	$m$,	similar to the one developed in Subsubsection	\ref{sub-estim-pen}.	Next,	we	conclude	by	using	the	same	arguments	of	Subsubsection	\ref{limit-passage-pen}.\\

We	have
$\dfrac{dy_m}{dt}+A(y_m,\cdot)-G(\cdot,y_m)\phi_m-\rho_m
=f.
$	Let	$0\leq	t\leq	T$,
by using $y_m$ as a test function and integrating in time from $0$ to $t$, we get
\begin{align*}
\|y_m(t)\|_{H}^2+2\displaystyle\int_0^t \langle A(y_m,\cdot),y_m\rangle ds &-2\int_0^t(G(\cdot,y_m)\phi_m,y_m)ds-2\int_0^t \int_D\rho_my_mdxds\\
&\qquad = \|u_0\|^2_{H}+\int_0^t\langle	f,y_m\rangle	ds.
\end{align*}
Note that,	by	using	Lewy-Stampacchia	 	and	Young	inequalities,	for	any	$\gamma>0$
\begin{align*}
2\vert\int_0^t \int_D\rho_my_mdxds\vert&\leq
2\int_0^t \!\!\!\int_D\vert\rho_m\vert	\vert	y_m\vert	dxds \\&\leq	\gamma\Vert	y_m\Vert_{L^{\tilde{q}}([0,t]\times	D)}^{\tilde{q}}+C_\gamma\Vert	h^-\Vert_{L^{\tilde{q}^\prime}([0,T]\times	D)}^{\tilde{q}^\prime}\leq	C_\gamma+\gamma\Vert	y_m\Vert_{L^{\tilde{q}}([0,t]\times	D)}^{\tilde{q}}.	\end{align*}
By	using	$H_3$,	one	has
$
	2\displaystyle\int_0^t(G(\cdot,y_m)\phi_m,y_m)ds
		\leq	Lt+\int_0^t(\Vert	\phi_m\Vert_{H_0}^2+L)\Vert	y_m\Vert_H^2ds.
$
 Thus, arguments	already	detailed	in	the	proof	of	Lemma	\ref{lem1} yield the	existence	of	$C>0$,	independent	of	$m$,	such	that
\begin{align*}
&\|y_m(t)\|_{H}^2+\alpha\displaystyle \int_0^t  \Vert y_m(s)\Vert_V^p  ds  \leq 
C \int_0^t(1+\Vert \phi_m\Vert_{H_0}^2) \Vert y_m(s) \Vert_H^2  ds+ \frac{\alpha}{2}\int_0^t\|y_m(s)\|^p_V ds+ C,
\end{align*}
Since	$(\phi_m)_m$	is	bounded	in	$L^2([0,T];H_0)$,	we	obtain
\begin{lemma}\label{lem1-m}
	\begin{itemize}
		\item $(y_m)_{m}$ is bounded in  $ L^p(0,T;V)\cap C([0,T];H)$.
		\item $(A(y_m,\cdot))_{m}$ is bounded in $L^{p^\prime}(0,T;V^\prime)$	and 	$(\rho_m)_m$	is	bounded	in	$L^{\tilde{q}^\prime}([0,T]\times D).$			\end{itemize}
	\end{lemma}
Similarly to the proof	of	Lemma	\ref{uniq-Ske}	with	$f_1=f_2$	and	using	that	$\phi_m$	converges	strongly	to	$\phi$	in	$L^2(0,T;H_0)$,	we	deduce
\begin{align*} 
	(u_m)_{m} 	\text{	is  a Cauchy sequence in the space in	}	C([0,T];H).
\end{align*}
Now,	we	are	in	position	to	use	the	same	arguments	as	in	the		Subsubsection	\ref{limit-passage-pen}	to	deduce	that	Theorem	\ref{THm-approx-1}	holds		with	$\phi\in	L^2([0,T];H_0)$.	  Lewy-Stampac\-chia inequality is a consequence of the passage to the limit in the one satisfied by $\rho_m$.
\subsubsection{The	general	case}\label{remove-h}	Let	$\phi\in	L^2([0,T];H_0)$
and $h^- \in (L^{p^\prime}(0,T;V^\prime))^+$. Thanks to Lemma \cite[Lemma	4.1]{GMYV}, there exists $h_n \in L^{\tilde{q}^\prime}(0,T; L^{\tilde{q}^\prime}(D))$ and	 non negative such that
$ h_n \longrightarrow h^-$ in	$L^{p^\prime}(0,T;V^\prime).$

Associated with $h_n$, denote the following $f_n$ by
$$
f_n=\dfrac{d\psi}{dt}+A(\psi,\cdot)+h^+-h_n, \quad h^+\in (L^{p^\prime}(0,T;V^\prime))^+ . 
$$
Note that $f_n \in L^{p^\prime}(0,T;V^\prime)$  and  converges strongly to $f$ in $ L^{p^\prime}(0,T;V^\prime)$.
Denote by $(y_n,k_n)$ the sequence of solutions given by Theorem \ref{THm-approx-1}	with	$\phi\in	L^2([0,T];H_0)$ where $h^-$ is replaced by $h_n$.
\\[0.1cm] 
By Lewy-Stampacchia inequality, one has $0\leq -k_n\leq h_n$. 
For any $\varphi \in L^p(0,T;V)$, it holds that
\begin{align*}
\displaystyle\int_0^T \vert \langle k_n, \varphi \rangle \vert ds &\leq \displaystyle\int_0^T \langle -k_n, \varphi^+ \rangle  ds+ \displaystyle\int_0^T  \langle - k_n, \varphi^- \rangle ds\\
&\leq  \displaystyle\int_0^T \langle h_n, \varphi^+ \rangle  ds+ \displaystyle\int_0^T  \langle h_n, \varphi^- \rangle ds \leq 2 \Vert h_n\Vert_{L^{p^\prime}(0,T;V^\prime)}\Vert \varphi\Vert_{L^{p}(0,T;V)}.
\end{align*}
Since $(h_n)_n$ converges to $h$ in $L^{p^\prime}(0,T;V^\prime),$ one gets that $(h_n)_n$ is bounded independently of $n$ in  $L^{p^\prime}(0,T;V^\prime)$ and therefore 
\begin{align}\label{bouded+Lagra}
		(k_n)_n \text{	is bounded independently of	} 	n \text{	in } L^{p^\prime}(0,T;V^\prime).
\end{align}

Let	$t\in [0,T]$	and $n\in \mathbb{N}^*$, by  using	  $y_n$	as	test	function	in	the	equations	satisfied	by	$y_n$, one gets  
\begin{align*}
\dfrac{1}{2}\Vert y_n(t) \Vert_H^2+\displaystyle\int_0^t\langle A(y_n,\cdot), y_n\rangle ds=\dfrac{1}{2}\Vert u_0 \Vert_H^2&+\displaystyle\int_0^t\langle -k_n, y_n\rangle ds+\displaystyle\int_0^t\langle f_n, y_n\rangle ds
+\displaystyle\int_0^t( G(y_n,\cdot)\phi, y_n)	ds 
\end{align*}
Since $f_n$ converges to $f$ in $L^{p^\prime}(0,T;V^\prime),$  it holds that $(f_n)_n$ is bounded independently of $n$ in  $L^{p^\prime}(0,T;V^\prime)$. Therefore, by Young's inequality,  we get $$ \displaystyle\int_0^T\vert\langle f_n-k_n, y_n\rangle\vert ds \leq \dfrac{\alpha}{2}\displaystyle\int_0^T\Vert y_n(s)\Vert_V^pds+C\Vert f_n-k_n\Vert_{L^{p^\prime}(0,T;V^\prime)}^{p^\prime}.$$
By using $H_3$, we get
$
(G(\cdot,y_n)\phi,y_n) \leq
(L+\Vert \phi\Vert_{H_0}^2)\Vert y_n\Vert_H^2+L.
$
Therefore
$$ \displaystyle\sup_{s\in [0,t]}\Vert y_n(s) \Vert_H^2+\displaystyle\int_0^t\Vert y_n(s)\Vert_V^pds \leq C(1+  \int_0^t(1+\Vert \phi\Vert_{H_0}^2) \sup_{\tau \in[0,s]}\Vert y_n(\tau) \Vert_H^2 ds).$$  
By using Gr\"onwall's lemma	and	$H_{2,3}$, one concludes  that 
\begin{align}\label{bound-gle}
(y_n)_n \text{ and }(A(y_n,\cdot))_n	\text{	are  bounded in 	}L^p(0,T;V)\cap C([0,T];H)	 	\text{	and 	}L^{p^\prime}(0,T;V^\prime),\text{	resp.}
\end{align}

Now, by	using	similar	arguments	to	the	proof	of			Lemma	\ref{uniq-Ske}	with	$\phi_1=\phi_2$,	and	that	$f_n $   converges strongly to $f$ in $ L^{p^\prime}(0,T;V^\prime)$,	we	get
	\begin{align}\label{Cauchy-gle}
	(y_n)_{n} \text{	is  a Cauchy sequence in the space 	}C([0,T];H)
	\end{align}
	
	Next,	by	using	the	same	arguments	as	in	the		Subsubsection	\ref{limit-passage-pen}	(see	also	\cite[Subsect.	3.3]{YV22}),		
	we deduce  Theorem \ref{THm-approx-1} for general $f$.  At last,  Lewy-Stampac\-chia inequality is a consequence of the passage to the limit in the one satisfied by $k_n$,	which	completes	the	proof	of	Proposition	\ref{Prop-Skeleton}.

	\section{Proof	of	large	deviation	principle}\label{section-LDP-proof}
	\subsection{Continuity  of  skeleton equations	with respect to the  signals}
\begin{lemma}\label{LDP-cd2}	Assume	that		($H_1$)-($H_6$)	hold.
Let	$N<\infty$,	for	any	family $\{ v^\delta: \delta>0\} \subset S_N $	satisfying	that	$v^\delta$	converges	weakly	to	some	element	$v$	as	$\delta\to	0$,	we	have
\begin{itemize}
	\item			$g^0(\int_0^\cdot v_s^\delta	ds)$	converges	to	$g^0(\int_0^\cdot v_s	ds)$	in	the	space	$C([0,T];H)$,
	\item	If	moreover	$H_7$	holds,	then	
	\begin{align*}
		g^0(\int_0^\cdot v_s^\delta	ds)	\text{	converges	to	}	g^0(\int_0^\cdot v_s	ds)	\text{	in	}		L^p(0,T;V)\cap	C([0,T];H),
	\end{align*}
	\end{itemize}
where	$g^0$	is	given	in	\eqref{g0-def}.	
		\end{lemma}
	\begin{proof}
Let	$(\phi_n)_n	\subset	S_N$	such	that		$\phi_n$	converges	to	$\phi$	weakly	in	$L^2(0,T;H_0)$.	Denote	by	$y_n$	and	$y^\phi$	the	solutions	of	\eqref{Skeleton}	corresponding	to	$\phi_n$	and	$\phi$	respectively.	To	prove	Lemma	\ref{LDP-cd2},	we	show	
$$	y_n	\to	y^\phi	\text{	in	}	L^p(0,T;V)\cap	C([0,T];H),$$
if	$H_1$-$H_7$	hold	and	the	convergence	in	$C([0,T];H)$		follows similarly,	if	$H_7$	does	not	hold.
	\subsubsection*{The	first	step}
	Thanks	to	Proposition	\ref{Prop-Skeleton},	there	exists	$(y_n,R_n)$	satisfying	
			\begin{itemize}
		\item[$\bullet$]  $(y_n,-R_n) \in (L^p(0,T;V)\cap C([0,T];H))\times (L^{p^\prime}(0,T;V^\prime))^+ $, 
		$y_n(0)=u_0 $ and $ y_n \geq \psi .$
		\item[$\bullet$] For any 
		$  v \in L^p(0,T;V);\quad v\geq \psi$,  $ \quad\langle R_n,y_n-v\rangle ds  \geq 0$ a.e. in $[0,T]$. 
		\item[$\bullet$] For all $t\in[0,T]$:	$y_n(t)+\displaystyle\int_0^tR_n ds+\displaystyle\int_0^tA(y_n,\cdot) ds=u_0+\int_0^t(f+G(\cdot, y_n)\phi_n) ds	$	in	$V^\prime$. 

		\item	The	following	Lewy	Stampacchia	inequality	holds
		\begin{align}\label{LS-bound}
	0\leq  \dfrac{dy_n}{dt}+ A(y_n,\cdot)- G(y_n,\cdot)\phi-f=-R_n \leq h^-=  \left(f - \partial_t  \psi  - A(\psi,\cdot)\right) ^-.
		\end{align} 
	\end{itemize}
	
First,	let	us	establish	some	uniforom	estimates	w.r.t.	$n$.	We	have
	\begin{align}\label{eqn-condition2}
	&\dfrac{dy_n}{dt}+A(y_n,\cdot)-G(\cdot,y_n)\phi_n+R_n
	=f
	\end{align}
	Let	$t\in[0,T]$,	by using $y_n$ as a test function and integrating in time from $0$ to $t$, we get
	\begin{align*}
	&\|y_n(t)\|_{H}^2+2\displaystyle\int_0^t \langle A(y_n,\cdot),y_n\rangle ds -2\int_0^t(G(\cdot,y_n)\phi_n,y_n)ds+2\int_0^t \langle	R_n,y_n\rangle	ds = \|u_0\|^2_{H}+2\int_0^t\langle	f,y_n\rangle	ds
	\end{align*}
	By	using	 \eqref{LS-bound},	we	obtain	that:	
	$
	(R_n)_n	\text{	is	bounded	by	}	\Vert	h^-\Vert_{L^{p^\prime}(0,T;V^\prime)}:=C	\text{	in	}	L^{p^\prime}(0,T;V^\prime).$	Thus,		by	using	
		Young	inequality,	for	any	$\gamma>0$
	\begin{align*}
	2\vert\int_0^t \langle	R_n,y_n\rangle	ds\vert&\leq
	2\int_0^t \!\!\!\Vert	R_n\Vert_{V\prime}	\Vert	y_n\Vert_V	ds \leq	C_\gamma+\gamma\Vert	y_n\Vert_{L^{p}(0,T;V)}^{p}.	\end{align*}
	By	using	$H_3$,	one	has	
	$\displaystyle\int_0^t(G(\cdot,y_n)\phi_n,y_n)ds	
		\leq	Lt+\int_0^t(\Vert	\phi_n\Vert_{H_0}^2+L)\Vert	y_n\Vert_H^2ds
.$	Therefore,	arguments	already	detailed	in	the	proof	of	Lemma	\ref{lem1}	ensure
		\begin{align*}
	&\|y_n(t)\|_{H}^2+\alpha\displaystyle\int_0^t  \Vert y_n(s)\Vert_V^p  ds  \leq 
	C \int_0^t(1+\Vert \phi_n\Vert_{H_0}^2) \Vert y_n(s) \Vert_H^2  ds+ \frac{\alpha}{2}\int_0^t\|y_n(s)\|^p_V ds+ C(f,T).
	\end{align*} Grönwall	inequality	ensures
$
\|y_n(t)\|_{H}^2+\displaystyle \int_0^t \frac{\alpha}{2} \Vert y_n(s)\Vert_V^p  ds  \leq 
	C(f,T) e^{T+N}.
$	Thus,	we	get

	\begin{align}\label{lem1-m-}
\left\{ \begin{array}{l}
 (y_n)_{n}	\text{ is bounded in	}   L^p(0,T;V)\cap C([0,T];H),\\[0.7ex]
(A(y_n,\cdot))_{n}	\text{	and	}		(R_n)_n \text{ are bounded in	} 	L^{p^\prime}(0,T;V^\prime),\\[0.7ex]	
	(G(\cdot,y_n)\phi_n)_n	\text{ is bounded in	}	L^2(0,T;H).
\end{array}
\right.
\end{align}
	
From	\eqref{eqn-condition2},	we	have
	$
	\dfrac{dy_n}{dt}=f-A(y_n,\cdot)+G(\cdot,y_n)\phi_n-R_n
	$
		and	$(\dfrac{dy_n}{dt})_n$	is	bounded	in	$L^{r}(0,T;V^\prime)$	where	$	r=\min(2,p^\prime),$	thanks	to		\eqref{lem1-m-}.	Set	
		$	\mathbb{V}=\{v\in	L^p(0,T;V);	\quad	\dfrac{dv}{dt}\in	L^{r}(0,T;V^\prime)	\}.$
		Thanks	to	\cite[Corollary	6]{Simon},	note	that	$\mathbb{V}\cap	L^\infty(0,T;H)$	is	compactly	embedded	into	$L^{q}(0,T;	H),$	for	any	finite	$	q>1$.
			By	using	 \eqref{lem1-m-}, 
			there exist $y \in L^\infty(0,T;H)\cap	\mathbb{V} $ and $ (R,\chi ) \in  (L^{p^\prime}(0,T;V^\prime))^2$ such that, up to sub-sequences denoted by the same way,
			\begin{align}
			y_n &\rightharpoonup y \quad \text{in} \quad  \mathbb{V}	\text{	and 	}	y_n\to	y	\text{	in} \quad	L^2(0,T;H),\label{3-8-Sk-}\\
			y_n &	\overset{*}{\rightharpoonup} y  \quad \text{in} \quad  L^\infty(0,T;H),\label{4-Sk-}\\ 
			(A(y_n,\cdot),R_n) &\rightharpoonup (\chi,R)  \quad \text{in} \quad L^{p^\prime}(0,T;V^\prime)\times	L^{p^\prime}(0,T;V^\prime),\label{5-Sk-}\\
			y_n(t) &\rightharpoonup y(t),	\forall	t\in[0,T]\quad  \text{in} \quad H.\label{7-Sk-}
			\end{align}
				Indeed,
			by  compactness  with respect to the weak topology in the spaces $ \mathbb{V}$,	$L^{p^\prime}(0,T;V^\prime)$ 	and	the compact embedding of	$\mathbb{V}\hookrightarrow	L^2(0,T;H)$,  there exist $y \in \mathbb{V}$, $\chi,R \in L^{p^\prime}(0,T;V^\prime)$  such that $\eqref{3-8-Sk-}$	and $\eqref{5-Sk-}$	hold (up	to	a  sub-sequences).	Moreover,	\eqref{4-Sk-}	follows	from
			the	compactness	with respect to the weak-* topology	in	$L^\infty(0,T;H).$ 
			Since	$\mathbb{V}	\hookrightarrow	C([0,T];V^\prime)$,	one	obtains	that	$y_n(t)	\rightharpoonup	y(t)	$	in	$V^\prime$	and	by	taking	into	acount	the	boundedness	of	$(y_n)_n$	in	$C([0,T];H)$,	we	obtain	\eqref{7-Sk-}.			On	the	other	hand, 
			$-R \in	(L^{p^\prime}(0,T;V^\prime))^+$ since the set of non positive elements of $L^{p^\prime}(0,T;V^\prime)$ is a closed convex subset of $L^{p^\prime}(0,T;V^\prime)$.	On	the	other	hand,	note	that
			$y(0)=u_0$		thanks	to		\eqref{7-Sk-},	since	  $y_n(0)=u_0$	and		$y$	satisfies	the	constaint\textit{	i.e.}	$y\geq\psi$	thanks	to	\eqref{3-8-Sk-}.
		As	a	consequence	of	\eqref{3-8-Sk-},	we	get	
			$$\lim_{n\to \infty}\int_0^t(G(\cdot, y_n)\phi_n,\Phi)ds=\int_0^t(G(\cdot, y)\phi,\Phi)ds,	\quad	\forall	\Phi\in	V,	\forall	t\in[0,T]. $$
			Indeed,	let	$\Phi\in	V$	and	note	that	
			\begin{align*}
		&\vert\int_0^t(G(\cdot, y_n)\phi_n-G(\cdot, y)\phi,\Phi)ds\vert\\&\leq	\vert\int_0^t(G(\cdot, y_n)\phi_n-G(\cdot, y)\phi_n,\Phi)ds+\int_0^t(G(\cdot, y)\phi_n-G(\cdot, y)\phi,\Phi)ds\vert\\
			&\leq	\int_0^t\Vert	G(\cdot, y_n)-G(\cdot, y)\Vert_{L_2(H_0,H)}\Vert\phi_n\Vert_{H_0}\Vert	\Phi\Vert_Hds+\vert\int_0^t(G(\cdot, y)\phi_n-G(\cdot, y)\phi,\Phi)ds\vert=I_1^n+I_2^n.	
					\end{align*}
					Thanks	to	Holder	inequality,	we	write
					\begin{align*}
						I_1^n\leq	C_\Phi\int_0^T\Vert	G(\cdot, y_n)-G(\cdot, y)\Vert_{L_2(H_0,H)}\Vert\phi_n\Vert_{H_0}ds	&\leq	\sqrt{N\cdot	M}C_\Phi(\int_0^T\Vert	y_n-y\Vert_{H}^2dt)^{\frac{1}{2}}\to	0.
					\end{align*}
		Recall	that	$G(\cdot, y)\in	L^2(0,T;L_2(H_0,H))$,		since	$\phi_n$	converges	weakly	to	$\phi$	in	$L^2(0,T;H_0)$	we	obtain	that	$\displaystyle\lim_{n\to \infty}	I_2^n=0$.	By	passing	to	the	limit	in	\eqref{eqn-condition2},	we	get
			$$(y(t),\Phi)+\displaystyle\int_0^t\langle R,\Phi\rangle ds+\displaystyle\int_0^t \langle \chi, \Phi\rangle ds=(u_0,\Phi)+\int_0^t(f+G(\cdot, y)\phi, \Phi) ds,\quad\forall	\Phi\in	V.$$ \\
		We	will	prove	that	 $A(y,\cdot)=\chi$	and	$\langle R, \psi-y\rangle ds=0$	a.e.	First,	we	take	the	difference	between	the	last	equation	and	\eqref{eqn-condition2}.\\
		 Let	$t\in[0,T]$,
				thanks	to	Proposition	\ref{Prop-Skeleton}, we	know	that	$\langle R_n, y_n-\psi\rangle=0$	a.e.	in	$[0,T]$,	then
		\begin{align*}
		 \displaystyle\int_0^t \langle R_n, y_n-y\rangle ds-\displaystyle\int_0^t \langle R, y_n-y\rangle ds &= \displaystyle\int_0^t \langle R_n, y_n-\psi\rangle	ds+\int_0^t \langle R_n, \psi-y\rangle ds-\displaystyle\int_0^t \langle R, y_n-y\rangle ds\\&= \int_0^t \langle R_n, \psi-y\rangle ds-\displaystyle\int_0^t \langle R, y_n-y\rangle ds.
		\end{align*}
			By	using	the	last	equality,	\eqref{eqn-condition2}	and	the	monotonicity	of	$A$	(see	Subsubsection	\ref{limit-passage-pen}),	we	obtain				\begin{align}
		\dfrac{1}{2}\Vert (y_n-y)(t)\Vert_H^2+	 \displaystyle\int_0^t \langle A(v,\cdot), y_n-v\rangle ds +\displaystyle\int_0^t \langle A(y_n,\cdot),v-y\rangle ds-\displaystyle\int_0^t \langle \chi, y_n-y\rangle ds \label{identi-LDP}\\+\displaystyle\int_0^t \langle R_n, \psi-y\rangle ds-\displaystyle\int_0^t \langle R, y_n-y\rangle ds \leq 	\vert\displaystyle\int_0^t \langle G(y_n,\cdot)\phi_n-G(y,\cdot) \phi, y_n-y\rangle ds\vert=\vert	I\vert,\notag\end{align}
	where	$v\in L^p(0,T;V)$.
				 	On	the	other	hand,	thanks	to	$H_3$	and	\eqref{lem1-m-},	we	get
			\begin{align*}
			\int_0^T	\Vert	G(y_n,\cdot)\phi_n\Vert_H^2ds&\leq	\int_0^T	\big\Vert	G(y_n,\cdot)\Vert_{L_2(H_0,H)}^2\Vert\phi_n\Vert_{H_0}^2ds\\
			&\leq	L\int_0^T	\big(1+\Vert	y_n\Vert_{H}^2)\Vert\phi_n\Vert_{H_0}^2ds\leq	C(L)\cdot	N.
			\end{align*}	
			and	similarly	we	get	$	\displaystyle\int_0^T	\Vert	G(y,\cdot)\phi\Vert_H^2ds\leq		C(L)\cdot	N$.
			 Thus,	by	using	\eqref{3-8-Sk-}	and	Cauchy	Schwarz	inequality,	we	deduce	$\displaystyle\lim_{n\to \infty}	\vert	I\vert=0$,	since
$			\vert	I\vert\leq	\int_0^T	\big(\Vert	G(y_n,\cdot)\phi_n\Vert_H+\Vert	G(y,\cdot)\phi\Vert_H\big)\Vert	y_n-y\Vert_Hds.
$
		Note	that	$\displaystyle\liminf_{n}\Vert (y_n-y)(t)\Vert_H^2=\liminf_{n}\Vert y_n(t)\Vert_H^2-\Vert y(t)\Vert_H^2	\geq	0,	\forall	t\in[0,T]$,	thanks	to
		\eqref{7-Sk-}.
By setting $t=T$	and	
		  passing to the  limit as $n \rightarrow \infty$	in	\eqref{identi-LDP},     we get  
		$$ \int_0^T\langle R, \psi-y\rangle ds+\displaystyle\int_0^T\langle A(v,\cdot)-\chi, y-v\rangle ds  \leq  0. $$
		By	setting	$v=y$	in	the	last	inequality,	we		get	$\int_0^T\langle R, \psi-y\rangle ds\leq	0$.	We	know	already	that	$-R\in	(L^{p^\prime}(0,T;V^\prime))^+$	and	$y\geq	\psi$,	which		gives	$\displaystyle\int_0^T\langle R, \psi-y\rangle ds\geq	0$.	Hence,	
			$\int_0^T\langle R, \psi-y\rangle ds=0$
		and	$\langle R, \psi-y\rangle	=0$	a.e.	in	$[0,T]$,	since	the	integrand	is	always	non	negative	(see		\cite[Remark	5]{YV22}).
		Finally,	one	has
			$ \displaystyle\int_0^T\langle A(v,\cdot)-\chi, y-v\rangle ds  \leq  0. $
	By	using "Minty's trick" \cite[Lemma 2.13 p.35]{Roubicek},	we	deduce 	that $A(y,\cdot)=\chi$	.
		In	conclusion,	$(y,R)$	satisfies	
				\begin{itemize}
			\item[$\bullet$]  $(y,-R) \in (L^p(0,T;V)\cap L^\infty(0,T;H))\times (L^{p^\prime}(0,T;V^\prime))^+ $, 
			$y(0)=u_0 $ and $ y \geq \psi .$
			\item[$\bullet$]	$\langle R, \psi-y\rangle	=0$	a.e.	in	$[0,T]$ and 
			$  \forall	v \in L^p(0,T;V):\quad v\geq \psi$,  $ \langle R,y-v\rangle   \geq 0$ a.e. in $[0,T]$. 
			\item[$\bullet$] For all $t\in[0,T]$, 
						\begin{align}\label{sket-lim}
			(y(t),\Phi)+\displaystyle\int_0^t\langle R,\Phi\rangle ds+\displaystyle\int_0^t \langle A(y,\cdot), \Phi\rangle ds=(u_0,\Phi)+\int_0^t(f+G(\cdot, y)\phi, \Phi) ds,\quad\forall	\Phi\in	V.
			\end{align} 

		\end{itemize}
	\subsubsection*{The	second	step}
	From	\eqref{3-8-Sk-},
		we	have	(up	to	subsequence,	\textit{a	priori})	
	\begin{align}\label{strong-cv-L2}
	y_n	\text{	converges	to	}	y		\text{	strongly	in	}	L^2(0,T;H),	
	\end{align}
	but	the	uniqueness	of	the	limit	yields	the	convergence	of	the	whole	sequence.
		Let	us	prove	that	the	following	convergence	holds
	\begin{align*}
		y_n	\to	y	\text{	in	}	L^p(0,T;V)\cap	C([0,T];H).
	\end{align*}
By	taking	the	difference	between	\eqref{eqn-condition2}	and	\eqref{sket-lim}	and	using	$y_n-y$	as	test	function,	we	get	
		\begin{align*}
	\dfrac{1}{2}\Vert &(y_n-y)(t)\Vert_H^2+\displaystyle\int_0^t \langle A(y_n,\cdot) -A(y,\cdot), y_n-y\rangle ds +\displaystyle\int_0^t \langle R_n-R, y_n-y\rangle ds\\
	&=\displaystyle\int_0^t \langle G(y_n,\cdot)\phi_n-G(y,\cdot) \phi, y_n-y\rangle ds=\mathcal{I}(t) ,\quad	\forall	t\in[0,T]. 
	\end{align*}
	Since	$y_n,	y\in	  L^p(0,T;V)$	and	satisfying	$ y_n,y\geq \psi$,	we	obtain	
	\begin{align*}
		\displaystyle\int_0^t \langle R_n-R, y_n-y\rangle ds=\displaystyle\int_0^t \langle R_n, y_n-y\rangle ds+\displaystyle\int_0^t \langle R, y-y_n\rangle ds\geq	0.
	\end{align*}
				Since	$\lambda_T Id+A$ is T-monotone, one	has $\displaystyle\int_0^t\langle A(y_n,\cdot)-A(y,\cdot),y_n-y\rangle ds \geq -\lambda_T\int_0^t\Vert y_n-y\Vert_H^2 ds.$
	 If	 moreover	H$_7$	holds, we  have
	$$\displaystyle\int_0^t\langle A(y_n,\cdot)-A(y,\cdot),y_n-y\rangle ds \geq -\lambda_T\int_0^t\Vert y_n-y\Vert_H^2 ds+\int_0^t	\bar{\alpha}\Vert y_n-y\Vert_V^pds.$$
	Therefore,
		\begin{align*}
	\dfrac{1}{2}\sup_{t\in [0,T]}\Vert &(y_n-y)(t)\Vert_H^2+\int_0^T	\bar{\alpha}\Vert (y_n-y)(s)\Vert_V^pds\\
	&\leq	\lambda_T\int_0^T\Vert (y_n-y)(s)\Vert_H^2 ds+\sup_{t\in [0,T]}\vert\displaystyle\int_0^t \langle G(y_n,\cdot)\phi_n-G(y,\cdot) \phi, y_n-y\rangle ds\vert,\quad	\forall	t\in[0,T]. 
	\end{align*}
	Similarly	to	$I$	(see	\eqref{identi-LDP}),	by	using	\eqref{strong-cv-L2}	one	has	
	$$\lim_{n\to \infty}\sup_{t\in [0,T]}\vert	\mathcal{I}(t)	\vert=\lim_{n\to \infty}\sup_{t\in [0,T]}\vert\displaystyle\int_0^t \langle G(y_n,\cdot)\phi_n-G(y,\cdot) \phi, y_n-y\rangle ds\vert=0.$$
	By	using	again	\eqref{strong-cv-L2},	we	conclude	
$
	y_n	\to	y	\text{	in	}	L^p(0,T;V)\cap	C([0,T];H).
$	Finally,
		from	Proposition	\ref{Prop-Skeleton},	$y$	is	the	unique	solution	to	\eqref{sket-lim}.	Thus,	$y=y^\phi$,	which	completes	the	proof	of	Lemma	\ref{LDP-cd2}.
	\end{proof}
	\subsection{On	vanishing	multiplicative	noise	$\delta\downarrow	0$}
	Let $\{\phi^\delta\}_{\delta>0} \subset \mathcal{A}_N$ for some $N<\infty$,	see	\eqref{An-LDP-2}. By using Girsanov theorem (\cite[Thm. 10.14]{Daprato-Zab92}) and	\cite[Prop.	10.17]{Daprato-Zab92},	we	obtain	the	existence	of		 a $Q$-Wiener process,	denoted	by	$W^\delta$, with respect to $\{\mathcal{F}_t\}_{t\geq 0}$ on the probability space $(\Omega,\mathcal{F},P^\delta)$ where  
		$ W^\delta(t)=W(t)+\dfrac{1}{\delta}\displaystyle\int_0^t\phi^\delta(s)ds, \quad t\in [0,T]$	and
		\begin{align}
	dP^\delta&=\exp[-\dfrac{1}{\delta}\int_0^T\langle \phi^\delta(s),dW(s)\rangle_{H_0}-\dfrac{1}{2\delta^2}\int_0^T\Vert \phi^\delta(s)\Vert_{H_0}^2ds]dP\label{equivP1}.
\end{align}
	In	fact,	the probability measure $P$  and $P^\delta$	are	mutually	absolutely 	continuous.
	By	using	Theorem	\ref{TH1}, there	exists	a	unique	 solution	$(v_\delta,r_\delta)$  satisfying
	\begin{align}\label{Eq-Girsanov}
	v_\delta=g^\delta(W(\cdot)+\dfrac{1}{\delta}\int_0^\cdot \phi^\delta(s)ds )=g^\delta(W^\delta(\cdot)).
	\end{align}

	\begin{enumerate}
		\item  $v_\delta$	is	$L^p(0,T;V)$-adapted 	process	with		P-a.s.		paths $v_\delta(\omega,\cdot) \in C([0,T];H).$	
		\item  $v_\delta(t=0)=u_0 \quad \text{and} \quad  v_\delta \geq \psi$	P-a.s.
		\item $P$-a.s, for all $t\in[0,T]$, 
		\begin{align}\label{v-delta}
		v_\delta(t)+\displaystyle\int_0^tr_\delta ds+\displaystyle\int_0^t A(v_\delta,\cdot)ds=u_0+\delta\displaystyle\int_0^t G(v_\delta,\cdot)dW(s)+\displaystyle\int_0^t fds+\int_0^tG(\cdot, v_\delta)\phi^\delta ds	\text{	in	}		V^\prime.
		\end{align}
		
		\item	$r_\delta$	is 	$L^{p^\prime}(0,T;V^\prime)$-adapted		process,	$- r_\delta\in (V^\prime)^+$	  and  $\langle r_\delta,v_\delta-\psi\rangle = 0$ \ a.e. in $\Omega_T$\footnote{	Note	that	$k\in	(L^p(\Omega,V^\prime))^+$	if	and	only	if	$k(t,\omega)	\in	(V^\prime)^+$	a.e.	in	$\Omega\times[0,T]$,	see	\cite[Remark	5]{YV22}.}.
	\end{enumerate}	
	By	using	Proposition	\ref{Prop-Skeleton},			there	exists	a	unique	random	solution		  $(u_\delta,k_\delta)$ satisfying	P-a.s.:
	\begin{enumerate}
		\item  $(u_\delta,-k_\delta) \in (L^p(0,T;V)\cap C([0,T];H))\times (L^{p^\prime}(0,T;V^\prime))^+ $, 
		$u_\delta(0)=u_0 $ and $ u_\delta \geq \psi .$
	\item	$\langle k_\delta,u_\delta-\psi\rangle  = 0$ a.e. in $[0,T]\times\Omega$
		\item	$P$-a.s, for all $t\in[0,T]$, 
		\begin{align}\label{u-delta}
		u_\delta(t)+\displaystyle\int_0^t k_\delta ds+\displaystyle\int_0^t  A(u_\delta,\cdot) ds=u_0+\int_0^t(f+G(\cdot, u_\delta)\phi^\delta) ds	\text{	in	}		V^\prime.
		\end{align}
	\end{enumerate}
		Moreover,		by	using		$g^0$	defined	by	\eqref{g0-def},	we	write	$u_\delta=g^0(\phi^\delta)$	P-a.s.
	
	\begin{lemma}\label{LDP-cd1}
		Assume	$H_1-H_6$.	Let $\{ \phi^\delta: \delta>0\} \subset \mathcal{A}_N $ for some $N<\infty$.  Then
		$$\lim_{\delta\to 0}P(\sup_{t\in [0,T]}\Vert	(v_\delta-u_\delta)(t)\Vert_H>\epsilon)=0,\quad\forall	\epsilon>0.$$
	Moreover,	if			$H_7$	holds,	then
	$\displaystyle\lim_{\delta\to 0}P(\vert	v_\delta-u_\delta\vert_T>\epsilon)=0,\quad\forall	\epsilon>0.$	
			\end{lemma}

	\subsubsection*{Proof	of	Lemma	\ref{LDP-cd1}}	First,	let	us	prove	some	uniform	estimates	with	respect	to	$\delta$.	
	For	that,	consider	$t\in[0,T]$,		by	using		\eqref{v-delta}	we	have
		\begin{align*}
				(v_\delta-\psi)(t)&+\displaystyle\int_0^tr_\delta ds+\displaystyle\int_0^t A(v_\delta,\cdot)ds=u_0-\psi(0)\\
				&\quad+\delta\displaystyle\int_0^t G(v_\delta,\cdot)dW(s)+\displaystyle\int_0^t (f-\dfrac{d\psi}{ds})	ds+\int_0^tG(\cdot, v_\delta)\phi^\delta ds.	
		\end{align*}
			Let	$\delta<1$,	by	using	Ito	formula	with	$F(u)=\Vert	u\Vert_H^2$	for	$v_\delta-\psi$,	we	get	
			\begin{align*}
			&\|(v_\delta-\psi)(t)\|_{H}^2+2\displaystyle\int_0^t \langle A(v_\delta,\cdot),v_\delta-\psi\rangle ds 
			+2\int_0^t\langle	r_\delta,v_\delta-\psi\rangle	ds\\
			& = \|u_0-\psi(0)\|^2_{H}+\delta^2\int_0^t\Vert	G(\cdot,v_\delta)\Vert_{L_2(H_0,H)}^2ds
			+
			2\int_0^t(G(\cdot,v_\delta)\phi^\delta,v_\delta-\psi)ds\\ &\qquad+
			2\int_0^t \langle  f- \dfrac{d\psi}{ds}  , v_\delta-\psi \rangle ds+2\delta\displaystyle\int_0^t (G(\cdot,v_\delta),v_\delta-\psi)dW(s).			
			\end{align*}
			By using $H_3$, we get
			\begin{align*}
			(G(\cdot,v_\delta)\phi^\delta,v_\delta-\psi)+\Vert	G(\cdot,v_\delta)\Vert_{L_2(H_0,H)}^2 &\leq \Vert G(\cdot,v_\delta)\Vert_{L_2(H_0,H)}\Vert \phi^\delta\Vert_{H_0}\Vert v_\delta-\psi\Vert_H+\Vert	G(\cdot,v_\delta)\Vert_{L_2(H_0,H)}^2	\\
			&\leq	C(L)(1+\Vert	\phi^\delta\Vert_{H_0}^2)\Vert v_\delta\Vert_H^2+C(L)(1+\Vert \psi\Vert_H^2+	\Vert	\phi^\delta\Vert_{H_0}^2).
			\end{align*}
		Argument detailed in the proof of Lemma 	\ref{lem1}	yields
	
			\begin{align*}
			\langle A(v_\delta,\cdot),v_\delta-\psi\rangle  
						\geq& \frac{\alpha}{2} \Vert v_\delta\Vert_V^p - \lambda \Vert v_\delta\Vert_H^2-l_1  - C(\psi),
			\end{align*}
			where $C(\psi) \in L^1([0,T])$.	Let	$\beta>0$,	by	using	Burkholder-Davis-Gundy	inequality	we	get
			\begin{align*}
				\E\sup_{r\in [0,t]} \left|\int_0^r (G(v_\delta,\cdot),v_\delta-\psi)_H \,dW \right|	&\leq	\beta\E\sup_{r\in [0,t]}\|(v_\delta-\psi)(r)\|_{H}^2+\dfrac{C}{\beta}\E\int_0^t\Vert G(\cdot,v_\delta)\Vert_{L_2(H_0,H)}^2ds\\
				&\leq	\beta\E\sup_{r\in [0,t]}\|(v_\delta-\psi)(r)\|_{H}^2+C_\beta	Lt+C_\beta	L\E\int_0^t\Vert v_\delta(s)\Vert_{H}^2ds.
			\end{align*}
						We	recall	that	 $\langle r_\delta,v_\delta-\psi\rangle = 0$  a.e. in $\Omega_T$.
			 Thus, for any positive $\gamma$, Young's inequality yields the existence of a positive constant $C_\gamma$ that may change form line to line, such that 
			\begin{align*}
			&(1-\beta)\E\sup_{r\in [0,t]}\|(v_\delta-\psi)(r)\|_{H}^2+2\displaystyle \E\int_0^t \frac{\alpha}{2} \Vert v_\delta(s)\Vert_V^p  ds  \leq  
			 \E\int_0^t[\lambda+C(L,\beta)(1+\Vert	\phi^\delta\Vert_{H_0}^2)] \Vert v_\delta(s)\Vert_H^2  ds\\&	\qquad+ \|l_1+C(\psi)\|_{L^1([0,T])} + C_\gamma(f, \dfrac{d\psi}{dt} ) 
		+ \gamma \E\int_0^t\|(v_\delta-\psi)(s) \|^p_V ds+C(L,\beta)(t+N). 
			\\&\qquad\qquad\leq 
			 C\E\int_0^t (1+\Vert	\phi^\delta\Vert_{H_0}^2)\Vert (v_\delta - \psi)(s)\Vert_H^2  ds+ \frac{\alpha}{2}\E\int_0^t\|v_\delta(s)\|^p_V ds+ C(N,\psi,f),
			\end{align*}
			for a suitable choice of $\gamma$. 	An	appropriate	choice	of	$\beta$	and	  Gronwall's lemma	yield	the	existence	of	$C:=C(\psi,f,N)>0$	independent	of	$\delta$	such	that		
					\begin{align}\label{boundedness-v-delta}
				\E\sup_{r\in [0,t]}\|v_\delta(r)\|_{H}^2+\alpha\displaystyle \E\int_0^t  \Vert v_\delta(s)\Vert_V^p  ds\leq	C(\psi,f,N),	\quad	\forall	t\in[0,T].
					\end{align}
					Moreover,	we	obtain	a	similar	estimate	for	$u_\delta$,	namely
										\begin{align}\label{boundedness-u-delta}
					\E\sup_{r\in [0,t]}\|u_\delta(r)\|_{H}^2+\alpha\displaystyle \E\int_0^t  \Vert u_\delta(s)\Vert_V^p  ds\leq	C(\psi,f,N),\quad	\forall	t\in[0,T].
					\end{align}
		Next,	\eqref{boundedness-v-delta}	and	\eqref{boundedness-u-delta}	will	be	used	to	prove	Lemma	\ref{LDP-cd1}.\\
			
				By	taking	the	difference between		\eqref{v-delta}	and	\eqref{u-delta},	we	have	$P$-a.s, for all $t\in[0,T]$, 
\begin{align*}
	v_\delta(t)-u_\delta(t)&+\displaystyle\int_0^t(r_\delta-k_\delta) ds+\displaystyle\int_0^t [A(v_\delta,\cdot)-A(u_\delta,\cdot)]ds\\&=\delta\displaystyle\int_0^t G(v_\delta,\cdot)dW(s)+\int_0^t[G(\cdot, v_\delta)\phi^\delta-G(\cdot, u_\delta)\phi^\delta] ds	\text{	in	}		V^\prime.
\end{align*}	
Let	$t\in	[0,T]$.	By	using	Ito	formula	with	$F(u)=\Vert	u\Vert_H^2$	for	$v_\delta-u_\delta$,	we	get
	\begin{align*}
&\|(v_\delta-u_\delta)(t)\|_{H}^2+2\displaystyle\int_0^t \langle A(v_\delta,\cdot)-A(u_\delta,\cdot),v_\delta-u_\delta\rangle ds 
+2\int_0^t\langle	r_\delta-k_\delta,v_\delta-u_\delta\rangle	ds \\&= \delta^2\int_0^t\Vert	G(\cdot,v_\delta)\Vert_{L_2(H_0,H)}^2ds
	+
2\int_0^t(G(\cdot,v_\delta)\phi^\delta-G(\cdot,u_\delta)\phi^\delta,v_\delta-u_\delta)ds\\&\qquad+
2\delta\displaystyle\int_0^t (G(\cdot,v_\delta),v_\delta-u_\delta)dW(s).
\end{align*}
Recall	that				$r_\delta$	and	$k_\delta$	satisfy	 
\begin{align}\label{reflected}
v_\delta,u_\delta\geq\psi,	\quad- r_\delta,-k_\delta\in (V^\prime)^+;\quad	\langle r_\delta,v_\delta-\psi\rangle  = 0 \text{	and 	}	\langle k_\delta,u_\delta-\psi\rangle  = 0\text{	a.e. in	} \Omega_T.	
\end{align}
Then	
$\displaystyle\int_0^t\langle	r_\delta-k_\delta,v_\delta-u_\delta\rangle	ds	\geq	0	\text{	a.e.	in	}	\Omega.$
	By	using	$H_{2,2}$, one	has $$\displaystyle\int_0^t\langle A(v_\delta,\cdot)-A(u_\delta,\cdot),v_\delta-u_\delta\rangle ds \geq -\lambda_T\int_0^t\Vert v_\delta-u_\delta\Vert_H^2 ds.$$
			 If	moreover H$_7$	holds, we  have
	$$\displaystyle\int_0^t\langle A(v_\delta,\cdot)-A(u_\delta,\cdot),v_\delta-u_\delta\rangle ds \geq -\lambda_T\int_0^t\Vert v_\delta-u_\delta\Vert_H^2 ds+\int_0^t	\bar{\alpha}\Vert v_\delta-u_\delta\Vert_V^pds.$$
By	using	$H_{3,1}$	and	Young's	inequality
\begin{align*}
2\vert\int_0^t(G(\cdot,v_\delta)\phi^\delta-G(\cdot,u_\delta)\phi^\delta,v_\delta-u_\delta)ds\vert&\leq	\int_0^t \Vert G(\cdot, v_\delta)-G(\cdot, u_\delta)\Vert_{L_2(H_0,H)} \Vert \phi^\delta\Vert_{H_0} \Vert v_\delta-u_\delta\Vert_Hds\\
&\leq\displaystyle\int_0^t (M+\Vert \phi^\delta\Vert_{H_0}^2)\Vert v_\delta-u_\delta\Vert_H^2 ds. 
\end{align*}
Thus
	\begin{align*}
\|(v_\delta-u_\delta)(t)\|_{H}^2&+\int_0^t	\bar{\alpha}\Vert (v_\delta-u_\delta)(s)\Vert_V^pds \leq 
\int_0^t (M+\lambda_T+\Vert \phi^\delta\Vert_{H_0}^2)\Vert (v_\delta-u_\delta)(s)\Vert_H^2ds\\
&+\delta^2\int_0^t\Vert	G(\cdot,v_\delta)\Vert_{L_2(H_0,H)}^2ds
+2\delta\vert\displaystyle\int_0^t (G(\cdot,v_\delta),v_\delta-u_\delta)dW(s)\vert
.
\end{align*}
Apply	Grönwall	inequality	to	deduce
\begin{align*}
\sup_{t\in [0,T]}\|(v_\delta-u_\delta)(t)\|_{H}^2&+\int_0^T	\Vert (v_\delta-u_\delta)(s)\Vert_V^pds \\
&\leq	\delta	C(T,N)[\int_0^T\Vert	G(\cdot,v_\delta)\Vert_{L_2(H_0,H)}^2ds
+\sup_{t\in [0,T]}\vert\displaystyle\int_0^t (G(\cdot,v_\delta),v_\delta-u_\delta)dW(s)\vert],
\end{align*}
where	$C(T,N)=e^{T(M+\lambda_T)+N}$.
By	using	$H_{3,2}$	and	\eqref{boundedness-v-delta},	there	exists	$C>0$	independent	of	$\delta$	such	that
$
\E	\displaystyle\int_0^T\Vert	G(\cdot,v_\delta)\Vert_{L_2(H_0,H)}^2ds		 \leq L T+\E	\int_0^T\Vert v_\delta\Vert_H^2ds	\leq	C.
$
Thanks	to	Burkholder-Davis-Gundy inequality		and	$H_{3,2}$,	there	exists	$C>0$	independent	of	$\delta$	such	that
	\begin{align*}
&\E\sup_{r\in [0,T]} \left|\int_0^r (G(v_\delta,\cdot),v_\delta-u_\delta)_H \,dW \right|	\leq	2\E\sup_{r\in [0,t]}\|(v_\delta-u_\delta)(r)\|_{H}^2+2\E\int_0^t\Vert G(\cdot,v_\delta)\Vert_{L_2(H_0,H)}^2ds\\
&\qquad\qquad\leq	2\E\sup_{r\in [0,T]}\|v_\delta(r)\|_{H}^2+2\E\sup_{r\in [0,T]}\|u_\delta(r)\|_{H}^2+2	LT+C	L\E\int_0^T\Vert v_\delta(s)\Vert_{H}^2ds\leq	C,	
\end{align*}
	thanks	to	\eqref{boundedness-v-delta}	and	\eqref{boundedness-u-delta}.
Therefore\begin{align*}
	\E\sup_{t\in [0,T]}\|(v_\delta-u_\delta)(t)\|_{H}^2&+\E\int_0^T	\Vert (v_\delta-u_\delta)(s)\Vert_V^pds \leq	\delta	C(T,N)	\to	0	\text{	as	}	\delta	\to	0.
\end{align*}
Finally,	Markov's inequality	ensures	the	conclusion	of	Lemma	\ref{LDP-cd1}.
\subsection*{Acknowledgements}
	This work is funded by national funds through the FCT - Funda\c c\~ao para a Ci\^encia e a Tecnologia, I.P., under the scope of the projects UIDB/00297/2020 and UIDP/00297/2020 (Center for Mathematics and Applications).


\end{document}